\documentclass[a4paper,12pt]{article}

\usepackage{amssymb}                                    
\usepackage{mathrsfs}                    
\usepackage{amsmath}                    
\usepackage{amsfonts}                   
\usepackage{amsthm}                     
\usepackage[latin1]{inputenc}           
\usepackage[arrow, matrix, curve]{xy}    
\usepackage[english]{babel}                
\usepackage{authblk}                 
\usepackage{graphics}
\usepackage{color}
\usepackage{hyperref}

\newtheoremstyle{introduction}
  {3 pt}
  {4 pt}
  {\itshape}
  {}
  {\bfseries}
  {:}
  {.5em}
  {}

\theoremstyle{plain}
\newtheorem{thm}{Theorem}[section]
\newtheorem{corollary}[thm]{Corollary}
\newtheorem{proposition}[thm]{Proposition}
\newtheorem{lemma}[thm]{Lemma}
\newtheorem{example}[thm]{Example}

\theoremstyle{introduction}
\newtheorem*{thm*}{Theorem} 

\theoremstyle{definition}
\newtheorem{definition}[thm]{Definition}
\newtheorem{remark}[thm]{Remark}

\newcommand\Hom		{\mathrm{Hom}}
\newcommand\dSet	{\text{dSet}}
\newcommand\Set     {\text{Set}}
\newcommand\sSet    {\text{sSet}}
\newcommand\Ab      {\mathrm{AbGr}}

\newcommand\sHom    {\underline{\text{sHom}}}

\newcommand\Symm    {\text{Sym}}
\newcommand\Oper    {\text{Oper}}

\newcommand\F       {\mathcal{F}}
\newcommand\E       {\mathcal{E}}
\newcommand\EsSet   {E_\infty\text{-spaces}}

\newcommand{\ds}{\displaystyle}
\def\to{\rightarrow}

\title{Dendroidal sets as models for connective spectra}
\author[1]{Matija Ba\v{s}i\'{c}}
\author[2]{Thomas Nikolaus}
\affil[1]{Prirodoslovno-matemati\v{c}ki fakultet, Sveu\v{c}ili\v{s}te u Zagrebu, Croatia}
\affil[2]{Fakultät für Mathematik, Universität Regensburg,  Germany}

\begin{document}
\maketitle
\renewcommand{\thefootnote}{\fnsymbol{footnote}} 
\footnotetext{\emph{MSC classes}: 55P47, 55P48, 18D50.}     
\renewcommand{\thefootnote}{\arabic{footnote}} 

\begin{abstract} \noindent
Dendroidal sets have been introduced as a combinatorial model for homotopy coherent operads. We introduce the notion of fully Kan dendroidal sets and show that there is a model structure on the category of dendroidal sets with fibrant objects given by fully Kan dendroidal sets. Moreover we show that the resulting homotopy theory is equivalent to the homotopy theory of connective spectra.
\end{abstract}

\section{Introduction}

The notion of a dendroidal set is an extension of the notion of a simplicial set, introduced to serve as a combinatorial model for $\infty$-operads \cite{MoerW07}.
The homotopy theory of $\infty$-operads is defined as an extension of Joyal's
homotopy theory of $\infty$-categories to the category of dendroidal sets. More precisely there is a class of dendroidal sets called inner Kan dendroidal sets (or simply $\infty$-operads) which are defined
analogously to inner Kan complexes (also known as $\infty$-categories) by lifting conditions \cite{MW09}. These objects form fibrant objects in a model structure on the category of dendroidal sets, which is
Quillen equivalent to coloured topological operads as shown in a series of papers by Cisinski and Moerdijk \cite{MC10, MC11, MC11a}.

The category of dendroidal sets behaves in many aspects similarly to the category of simplicial sets. One instance of this analogy is the model structure described above
generalizing the Joyal model structure. Another instance is the fact that there is a nerve functor from (coloured) operads into dendroidal sets generalizing the nerve functor from categories into simplicial sets. But there are two important aspects of the theory of simplicial sets that have not yet a counterpart in the theory of dendroidal sets:
\vspace{-0.3cm}
\begin{enumerate}
\setlength{\itemsep}{-0.5ex}
\item\label{pointKan} Kan complexes and the Kan-Quillen model structure on simplicial sets\footnote{In fact there is a model structure  constructed by Heuts \cite{Heuts1} that could be seen as a counterpart. We comment on this model structure later.}.
\item\label{pointGeo} The geometric realization of simplicial sets.
\end{enumerate}
\vspace{-0.3cm}
The two aspects are closely related since the geometric realization $|-|: \sSet \to \text{Top}$ is a left Quillen equivalence with respect to the Kan-Quillen model structure on simplicial sets. With respect to the Joyal model structure the geometric realization functor is still a left Quillen functor (but not an equivalence), as follows from the fact that the Kan-Quillen model structure is a left Bousfield localization of the Joyal model structure. The problem of finding counterparts for these structures in the theory of dendroidal sets has been raised almost with the introduction of dendroidal sets, see e.g. \cite[Section 5]{Weiss2010}. \\

In the present paper we construct analogues of \ref{pointKan} and \ref{pointGeo} for the category of dendroidal sets.
More precisely we introduce the notion of a fully Kan dendroidal set which (in analogy to a Kan complex in simplicial sets) has fillers for all horns of dendroidal sets and not just for inner horns (as for inner Kan dendroidal sets), see Definition \ref{deffullyKan}. As a first result we show that a certain subclass of fully Kan dendroidal sets, called strictly fully Kan dendroidal sets, spans a category equivalent to the category of Picard groupoids, Corollary  \ref{corollaryPic}. This already provides a hint as to what the geometric realization might be since it is well known that Picard groupoids model all connective spectra with vanishing $\pi_n$ for $n\geq 2$, \cite{may2008, Osorno2012}.

In fact, fully Kan dendroidal sets model all connective spectra. This is the main result of this paper:
\begin{thm*}[Theorems \ref{simplicial}, \ref{THMFIBRANT} and \ref{thm_equivalence}] ~
There is a model structure on dendroidal sets, called the \emph{stable model structure}, with fibrant objects given by fully Kan dendroidal sets which is a left Bousfield localization of the Cisinski-Moerdijk model structure. Moreover the stable model structure on dendroidal sets is Quillen equivalent to connective spectra.
\end{thm*}
The stable model structure has good formal properties, i.e. it is left proper, simplicial, tractable and combinatorial. Furthermore it allows for an explicit characterization of weak equivalences between fibrant objects.
The Quillen equivalence between dendroidal sets and connective spectra factors through the category of group-like $E_\infty$-spaces.

The proof of our theorem is based on constructions of Heuts \cite{Heuts1, Heuts2}. Heuts establishes a model structure on dendroidal sets, called the \emph{covariant model structure}, which lies between the Cisinski-Moerdijk model structure and the stable model structure.
Although we had at first obtained the stable model structure by different techniques, in this paper we construct it as a left Bousfield localization of the covariant model structure. This enables us to directly use another main result of Heuts: there is a Quillen equivalence between the covariant model structure and the model category of $E_\infty$-spaces. Our Quillen equivalence (Theorem \ref{thm_equivalence}) can then be derived by showing that the stable localization on the side of dendroidal sets corresponds to the group-like localization of $E_\infty$-spaces, see section \ref{sec:equivalence}. One disadvantage of this construction is that establishing the explicit description of fibrant objects is technically demanding, see sections \ref{partI} - \ref{partIII}. \\

Finally we want to mention that our results not only show that fully Kan dendroidal sets form a model for Picard $\infty$-groupoids but also that the $\infty$-category of Picard $\infty$-groupoids is a full reflective subcategory (in the sense of Lurie \cite[Remark 5.2.7.9]{HTT}) of the $\infty$-category of $\infty$-operads. The functor associating a spectrum to a dendroidal set will be further investigated in \cite{NikKtheory} and related to the geometric realization of simplicial sets. \\

\noindent {\bf Acknowledgements.} The authors thank Gijs Heuts, Ieke Moerdijk, and Markus Spitzweck for helpful discussion and Konrad Waldorf for comments on the draft. Special thanks to Urs Schreiber for the suggestion to look at fully Kan dendroidal sets in order to find a geometric realization.
The first author would also like to thank the Croatian Science Foundation for financial support and the Radboud University Nijmegen for its kind hospitality during the period in which this article was written. 

\section{Dendroidal sets and model structures}
In this section we will review some facts from the theory of dendroidal sets without always giving explicit references. For more details we refer the reader to the lecture notes \cite{MoerBar} and the papers \cite{MoerW07, MW09}.

First, we briefly recall the definition of the category of dendroidal sets. It is based on the notion of trees. 
A (finite rooted) tree is a non-empty connected finite graph with no loops equipped with a distinguished outer edge called the root and a (possibly empty) set of 
outer edges not containing the root called leaves. By convention, the term vertex of a tree refers only to non-outer vertices. 
Each tree $T$ generates a symmetric, coloured operad $\Omega(T)$ (in the category of sets) which has the edges of $T$ as colours and a generating operation for every vertex.
Using this construction we can define the category $\Omega$ which has finite rooted trees as objects and the set of morphisms between trees 
$T$ and $T'$ is given by the set of operad maps between operads $\Omega(T)$ and  $\Omega(T')$. 
Similarly to the definition of simplicial sets we define the category of dendroidal sets as the presheaf category on $\Omega$:
\begin{equation*}
\dSet := [\Omega^{op},\Set].
\end{equation*}
The dendroidal set represented by a tree $T$ is denoted by $\Omega[T]$. In particular for the tree $~|~$ with one edge which is also a leaf and a root, we set $\eta:= \Omega[~|~]$. The inclusion of $\Omega$ into the category of coloured, symmetric operads induces a fully faithful functor $N_d: \Oper \to \dSet$ called the dendroidal nerve.
We have $N_d(\Omega(T)) = \Omega[T]$.

There is a fully faithful embedding of the simplex category $\Delta$ into $\Omega$ by considering finite linear ordered sets as linear trees. This induces an adjunction
\begin{equation*}
i_!: \xymatrix{\sSet \ar@<0.3ex>[r] & \dSet: i^* \ar@<0.7ex>[l]}
\end{equation*}
where the left adjoint $i_!$ is fully faithful (there is also a further right adjoint $i_*$ which does not play a role in this paper).

The theory of dendroidal sets behaves very much like the theory of simplicial sets. In particular, for each tree $T$ there is a familiy of subobjects of $\Omega[T]$ in $\dSet$ called faces. There are two types of faces: {the} inner faces which are labeled by 
{the} inner edges of $T$ and {the} outer faces which are labeled by {the} vertices of $T$ with exactly one inner edge attached to it. The boundary $\partial \Omega[T]$ of $\Omega[T]$ is by definition the union of all faces of $T$. A horn is defined as the union of 
all but one {face}, see \cite{MW09} or \cite{MoerBar}. We distinguish inner and outer horns and we write $\Lambda^a[T]$ where $a$ is an inner edge or an outer vertex of $T$.

\begin{definition}
Let $T$ be a tree with at least 2 vertices. We call a horn $\Lambda^{a}[T] \subset \Omega[T]$ a \emph{root horn}, if $a$ is the unique vertex attached to the root.

The corolla $C_n$ is the tree with one vertex and $n$ leaves. There are $n+1$ faces of a corolla $C_n$, one for each colour (edge). The horns are the unions of all but one colour, denoted by $\Lambda^a[C_n]$ 
{where} $a$ {is} the omitted colour. We call this horn a \textit{leaf horn} if $a$ is the root (i.e. the leaf horn is the inclusion of the leaves) and a \textit{root horn} otherwise. 
\end{definition}

Note that most trees do not have a root horn. A root horn can only exist, if the tree is a corolla or the whole tree is concentrated over a single leaf of the root vertex.  

\begin{definition}\label{def_inner}
A dendroidal set $D$ is called \emph{inner Kan} if it admits fillers for all inner horns, 
i.e. for any inner edge $e$ of a tree $T$ and a morphism $\Lambda^{e}[T] \to D$ there is a morphism $\Omega[T] \to D$ that renders the diagram
\begin{equation*}
\xymatrix{
\Lambda^{e}[T] \ar[d]\ar[r] & D \\
\Omega[T] \ar[ru] &
}
\end{equation*}
commutative. A \emph{dendroidal Kan complex} is a dendroidal set that admits fillers for all horns {that are not root horns}.
\end{definition}

The class of inner Kan dendroidal sets has been introduced and studied in \cite{MW09, MC10} and the class of dendroidal Kan complexes in \cite{Heuts2}. The main results are

\begin{thm}[Cisinski-Moerdijk]
There is a left proper, combinatorial model structure on the category of dendroidal sets with cofibrations given by normal monomorphisms and fibrant objects given by inner Kan dendroidal sets. This model category is Quillen equivalent to the model category of coloured topological operads.
\end{thm}

\begin{thm}[Heuts]
There is a simplicial left proper, combinatorial model structure on the category of dendroidal sets with cofibrations given by normal monomorphisms and fibrant 
objects given by dendroidal Kan complexes. This model structure is called the \emph{covariant model structure} and is Quillen equivalent to the 
{standard}
model category of $E_\infty$-spaces.
\end{thm}

The slogan is that inner Kan dendroidal sets are a combinatorial model for topological operads and dendroidal Kan complexes are a 
model for $E_\infty$-spaces. The weak equivalences are called operadic equivalences in the Cisinski-Moerdijk model structure and 
covariant equivalences in the Heuts model structure. Note in particular that the covariant model structure is simplicial in 
contrast to the Cisinski-Moerdijk model structure. The simplicial enrichment in question is induced by the Boardman-Vogt
type tensor product on the category $\dSet$, see \cite{MW09}.

\section{Fully Kan dendroidal sets }\label{sec:22}

Similarly to the Definition \ref{def_inner} of inner Kan dendroidal sets we give the following definition.
\begin{definition}\label{deffullyKan}
A dendroidal set $D$ is called \emph{fully Kan} if it has fillers for all horn inclusions. This means that for each morphism $\Lambda^a[T] \to D$ (where $a$ is an inner edge or an outer vertex) there is a morphism $\Omega[T] \to D$  rendering the diagram
$$\xymatrix{
\Lambda^a[T] \ar[r]\ar[d] & D \\
\Omega[T] \ar[ru] &
}$$
commutative. $D$ is called \emph{strictly fully Kan} if additionally all fillers for trees $T$ with more than one vertex are unique.
\end{definition} 

\begin{remark} 
\begin{itemize}
\item
A fully Kan dendroidal set is also a dendroidal Kan complex and an inner Kan dendroidal set.
\item
The reader might wonder why we do not impose uniqueness for corolla fillers in the strictly fully Kan condition. The reason is that this forces the underlying simplicial set to be discrete as we will see in Proposition \ref{prop:abelian}.
\end{itemize}
\end{remark}

Let $C$ be a (small) symmetric monoidal category. We can define a coloured operad $\mathcal{O}_C$ as follows. The colours are the objects of $C$. The set of $n$-ary operations is defined as
$$\mathcal{O}_C(c_1,...,c_n; c) := \Hom_C(c_1 \otimes ... \otimes c_n, c).$$
The $\Sigma_n$-action is induced by the symmetric structure on $C$ and the composition is given by composition in $C$. Note that the expression $c_1 \otimes ... \otimes c_n$ is strictly speaking not well-defined in a symmetric monoidal category. One can either make a choice of order in which to tensor (e.g. from left to right) or work with unbiased symmetric monoidal categories. These are symmetric monoidal categories which have not only two-fold, but also $n$-fold chosen tensor products. For a discussion of these issues see \cite[Chapter 3.3]{Lein04}.

We denote by $\Symm$ the category of symmetric monoidal categories together with lax monoidal functors. Recall that a lax monoidal functor $F: C \to D$ is a functor together with morphisms $F(c) \otimes F(c') \to F(c \otimes c')$ for each $c,c' \in C$ and $1 \to F(1)$ which have to satisfy certain coherence conditions but do not have to be isomorphisms. The construction described above gives a fully faithful functor
$$ \Symm \to \Oper. $$
By composing with the dendroidal nerve $N_d: \Oper \to \dSet$ for each symmetric monoidal category $C$ we obtain a dendroidal set which we denote by abuse of notation with $N_d(C)$.

In \cite{MW09} it is shown that a dendroidal set is strictly inner Kan if and only if it is of the form $N_d(P)$ for a coloured operad $P$. An analogous statement is true for strictly fully Kan dendroidal sets. Recall that a symmetric monoidal category is called a \emph{Picard groupoid} if its underlying category is a groupoid and its set of isomorphism classes is a group, i.e. there are `tensor inverses' for objects.
\begin{proposition}\label{prop:strict}
A dendroidal set $D$ is strictly fully Kan if and only if there is a Picard groupoid $C$ with $D \cong N_d(C)$.
\end{proposition}
\begin{proof}
First assume that $D$ is strictly fully Kan. Then {$D$ is, in particular, a strictly inner Kan dendroidal set} and \cite[Theorem 6.1]{MW09} shows that there is a coloured operad $P$ with $N_d(P) \cong D$. Let $C$ be the underlying category of $P$. Since the underlying simplicial set of $N_d(P)$ is a Kan complex we conclude that $C$ is a groupoid.

By \cite[Theorem 3.3.4]{Lein04} an operad $P$ comes from a unique symmetric monoidal category as described above if and only if for every sequence $c_1,...,c_n$ of 
objects in $P$ there is {a} universal tensor product, that is an object $c$ together with an operation $t\in P(c_1,...,c_n;c)$ such that for all objects $a_1,...,a_p, b_1,...,b_q, c'$ and operations $t'\in P(a_1,...,a_p, c_1,...,c_n, b_1,...,b_q; c')$ there is a unique element $s\in P(a_1,...,a_p, c, b_1,...,b_q; c')$ such that the partial composition of $s$ and $t$ in $P$ is equal to $t'$. A sequence $c_1,...,c_n$ of objects of $P$ determines a map from $\eta_{c_1} \sqcup \cdots \sqcup \eta_{c_n}$ to $N_d(P)$. Since $N_d(P)$ is fully Kan we can fill the horn $\eta_{c_1} \sqcup \cdots \sqcup \eta_{c_n}  \to \Omega[C_n]$ and obtain a morphism $\Omega[C_n] \to N_d(P)$. The root colour of this morphism provides an object $c$ in $P$ and the corolla provides an operation $t \in P(c_1,...,c_n; c)$. Assume we have another operation $t' \in P(a_1,...,a_p, c_1,...,c_n, b_1,...,b_q; c')$. Then we consider the tree $T$ which is given by

$$  
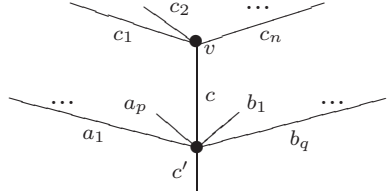
\xymatrix@R=10pt@C=12pt{
&&&&...&&&& \\
&&&*=0{~\bullet_v} \ar@{-}[ull]^{c_1} \ar@{-}[ul]_{c_2} \ar@{-}[urr]_{c_n} &&&&&\\
&...&_{a_p}&&_{b_1}&...&&&\\
&&&*=0{\bullet}\ar@{-}[ulll]^{a_1} \ar@{-}[ul] \ar@{-}[uu]_{c} \ar@{-}[ur] \ar@{-}[urrr]_{b_q}&&&&& \\
&&&*=0{}\ar@{-}[u]^{c'} &&&&&
}
$$
The operations $t$ and $t'$ provide a morphism $\Lambda^{v}[T] \to N_d{P}$, where $\Lambda^{v}[T]$ is the outer horn of $\Omega[T]$ at $v$. Since $D$ is strictly fully Kan we obtain a unique filler $\Omega[T] \to N_d(P)$, i.e. a unique $s\in P(a_1,...,a_p, c, b_1,...,b_q; c')$ with the sought condition. This shows that $c$ is the desired universal tensor product and that $P$ comes from a symmetric monoidal category.

The last thing to show is that $C$ is group-like. For $a$ and $c$ in $C$ we obtain an object $b$ together with a morphism $t \in P(a,b;c)$ by filling the {root} horn $\eta_a \sqcup \eta_c \to \Omega[C_2]$. But this is the same as a morphism  $a \otimes b \to c $ which is an isomorphism since $C$ is a groupoid. If we let $c$ be the tensor unit in $C$ then $b$ is the necessary inverse for $a$.  \\

Now assume conversely that $C$ is a Picard groupoid. Then the associated dendroidal set $N_d(C)$ admits lifts for corolla horns since tensor products and inverses exist (the proof is essentially the same as above). It remains to show that all higher horns admit unique fillers. To see this let $T$ be a tree with more than one vertex and $\Lambda^a[T]$ be any horn. A morphism $\Omega[T] \to N_d(C)$ is given by labeling the edges of $T$ with objects of $C$ and the vertices with operations in $C$ of higher arity, i.e. morphisms out of the tensor product of the  ingoing objects into the outgoing object of the vertex. The same applies for a morphism $\Lambda^a[T]\to N_d(C)$ where
the faces in the horn are labeled in the same manner {and} consistently.

The first observation is that for any labeling of the horn $\Lambda^a[T]$ already all edges of the tree $T$ are labeled, since the horn contains all colours of $T$ (for $T$ with more than one vertex). If the horn is inner then also all vertices of $T$ are already labeled if we label $\Lambda^a[T]$ and thus there is a unique filler.
If $a$ is an outer vertex and $T$ has more then two vertices then the same applies as one easily checks. Thus the horn can be uniquely filled.
Therefore we only have to deal with outer horns of trees with exactly two vertices. Such trees can all be obtained by grafting an $n$-corolla $C_n$ for $n \geq 0$ on top of a $k$-corolla for $k \geq 1$. We call this tree $C_{n,k}$.
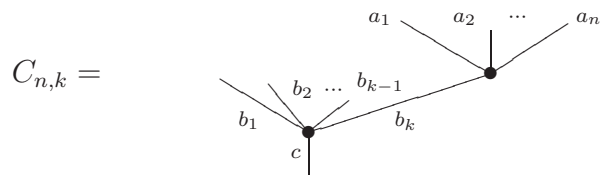
$$  \xymatrix@R=10pt@C=12pt{
&&&&&&&\\
&&&&_{a_1}&_{a_2 \quad\cdots}&_{a_n}& \\
C_{n,k}= \qquad && &&_{b_{k-1}}& *=0{\bullet} \ar@{-}[ul]\ar@{-}[u]\ar@{-}[ur]\ar@{-}&& \\
&&&*=0{\bullet}\ar@{-}[ull]^{b_1}\ar@{-}[ul]_{b_2 ~ \cdots}\ar@{-}[ur]\ar@{-}[urr]_{b_k}&&&&\\
&&&*=0{}\ar@{-}[u]^c &&&&
}
$$
 A morphism from the non-root horn $\Lambda^{v}[C_{n,k}] \to N_d(C)$ is then given by a pair consisting of a morphism $f: a_1 \otimes ... \otimes a_n \longrightarrow b_k$ and a morphism $g: b_1 \otimes ...  \otimes b_{k-1} \otimes a_1 \otimes ... \otimes a_n \longrightarrow c$ in $C$. Now we find a unique morphism $ g \circ (id \otimes f^{-1}): b_1 \otimes ... \otimes b_k \to c$ which renders the relevant diagram commutative, i.e. provides a filler $\Omega[C_{n,k}] \to N_d(C)$. A similar argument works for the case of the root horn of $C_{n,k}$. {This} finishes the proof.
\end{proof}

\begin{corollary}\label{corollaryPic}
The functor $N_d: \Symm \to \dSet$ induces an equivalence between the full subcategory of
Picard groupoids on the left and the full subcategory of
strictly fully Kan dendroidal sets on the right.
\end{corollary}
\begin{proof}
The functor $N_d$ is fully faithful since both functors $\Symm \to \Oper$ and $\Oper \to \dSet$ are. The restriction is essentially surjective by the last proposition.
\end{proof}
One of the main results of this paper shows that {a similar statement is valid} for fully Kan dendroidal sets that are not strict. They form a model for Picard $\infty$-groupoids, as we will show in the next sections. \\

Finally we want to give a characterization of strictly fully Kan {dendroidal sets for which} the corolla horns {also} admit unique fillers. Let $A$ be an abelian group, then we can associate to $A$ a symmetric monoidal category $A_{\text{dis}}$ which has $A$ as objects and only identity morphisms. The tensor product is given by the group multiplication of $A$ and is symmetric since $A$ is abelian. This construction provides a fully faithful functor from the category $\Ab$ of abelian groups to the category $\Symm$. Composing with the functor $\Symm \to \dSet$ constructed above we obtain a fully faithful functor
$$ i: \Ab \to \dSet.$$
Now we can characterize the essential image of $i$.
\begin{proposition}\label{prop:abelian}
For a dendroidal set $D$ the following two statements are equivalent
\begin{itemize}
\item $D$ is fully Kan with all fillers unique.
\item $D \cong i(A)$ for an abelian group $A$.
\end{itemize}
\end{proposition}
\begin{proof}
We already know by Proposition \ref{prop:strict} that strictly fully Kan dendroidal sets are of the form $N_d(C)$ for $C$ a Picard groupoid. We consider the underlying space $i^*D = NC$. This is now a strict Kan complex in the sense that all horn fillers are unique. In particular fillers for the horn $\Lambda^0[1] \to \Delta[1]$ are unique which shows that there are no non-degenerated 1-simplices in $NC$, hence no non-identity morphisms in $C$. Thus $C$ is a discrete category. But a discrete category which is a Picard groupoid is clearly of the form $A_{dis}$ for an abelian group $A$. This shows one direction of the claim. 
The other is {easier} and left to the reader.
\end{proof}
\section{The stable model structure}

{So far we have mentioned two model structures on dendroidal sets.}  In this section we want to describe another model structure on the category of dendroidal sets which we call the \emph{stable} model structure. 
We construct it as a left Bousfield localization {of the covariant model structure}. {Note that the covariant model structure is combinatorial and hence admits a left Bousfield localiztion with respect to any set of maps.} 
We will further explore {the stable} model structure to give a simple characterization of fibrant objects and weak equivalences.

The idea is to localize at a {root} horn of the 2-corolla
$$
\xymatrix@R=10pt@C=12pt{
&&&&\\
C_2=&&*=0{\bullet}\ar@{-}[ul]^{a}\ar@{-}[ur]_{b}&&\\
&&*=0{}\ar@{-}[u]^c.&&
}
$$
The relevant horn is given by the inclusion of the colours $a$ and $c$, i.e.
by the map
\begin{equation}\label{map:s}
s: {\Lambda^b [C_2]}  = \eta_a \sqcup \eta_c \longrightarrow \Omega[C_2].
\end{equation}
Note that there is also the inclusion of the colours $b$ and $c$ but this is essentially the same map since we deal with symmetric operads.

\begin{definition}
The \textit{stable} model structure on dendroidal sets is the left Bousfield localization of the covariant model structure at the map $s$. 
{Hence the} stable cofibrations are normal {monomorphisms between} dendroidal sets {and the stably} fibrant objects are those dendroidal Kan complexes $D$ for which the map
\begin{equation*}
s^*: \sHom(\Omega[C_2] , D) \to \sHom(\eta_a \sqcup \eta_c,D)
\end{equation*}
is a weak equivalence of simplicial sets.
\end{definition}

The general theory of left Bousfield localization yields the following:

\begin{thm}\phantomsection\label{simplicial}
\begin{enumerate}
\item
The category of dendroidal sets together with the stable model structure is a left proper, combinatorial, simplicial model category.

\item
The adjoint pair
$$ i_!: \xymatrix{\sSet \ar@<0.3ex>[r] & \dSet: i^* \ar@<0.7ex>[l]}$$
is a Quillen adjunction (for the stable model structure on dendroidal sets and the Kan-Quillen model structure on simplicial sets).
\item
The functor $i^*$ is homotopy right conservative, that is a morphism $f: D \to D'$ between stably fibrant dendroidal sets $D$ and $D'$ is a stable equivalence if and only if the underlying map
$i^*f: i^*D \to i^*D'$ is a homotopy equivalence of Kan complexes.
\end{enumerate}
\end{thm}
\begin{proof}
The first part follows from the general theory of Bousfield localizations {(see e.g. \cite[A.3]{HTT})}. 
{For the second statement,} note that the corresponding fact for the covariant model structure is true. 
Since the stable model structure is a left Bousfield localization of the covariant model structure, the claim 
follows by composition with the identity functor. The last assertion is true since a morphism between stably fibrant objects is a 
stable equivalence if and only if it is a covariant equivalence {and} covariant equivalences between fibrant objects
can be tested on the underlying spaces {(see \cite[Proposition 2.2.]{Heuts2})}. 
\end{proof} 


\begin{corollary}\label{criterion}
Let $f: X \to Y$ be a map of dendroidal sets. Then $f$ is a stable equivalence exactly if $i^*(f_K)$ is a weak equivalence where $f_K: X_K \to Y_K$ is the corresponding map between fully Kan (fibrant) replacements of $X$ and $Y$.
\end{corollary}

\begin{remark}
We could as well have localized at bigger collections of maps:
\begin{itemize}\setlength{\itemsep}{-0.5ex}
\item all corolla root horns
\item all outer horns
\end{itemize}
These localizations would yield the same model structure as we will see below. We decided to use only the 2-corolla in order to keep the localization (and the proofs) as simple as possible.
\end{remark}

As a next step we want to identify the fibrant objects in the stable model structure as the fully Kan dendroidal sets.  {First} we need some terminology:
\begin{definition}\label{ext_corolla}
{An} \emph{extended corolla} is {a} tree of the form
$$  
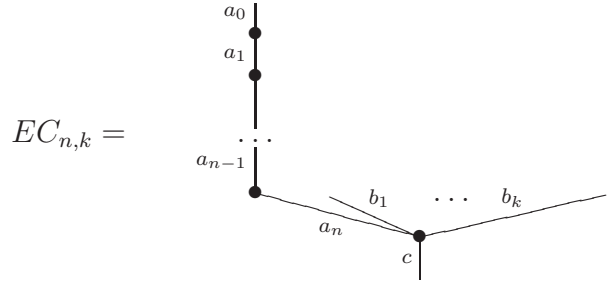
\xymatrix@R=10pt@C=12pt{
&&&&&&&\\
& *=0{\bullet} \ar@{-}[u]^{a_0} &&&&&& \\
& *=0{\bullet} \ar@{-}[u]^{a_1} &&&&&& \\
EC_{n,k} = \qquad & {\dots} \ar@{-}[u] &&&&&& \\
& *=0{\bullet} \ar@{-}[u]^{a_{n-1}} &&\qquad \dots &&&& \\
&&&*=0{\bullet}\ar@{-}[ul]_{b_1}\ar@{-}[ull]^{a_n}\ar@{-}[urrr]^{b_k}&&&&\\
&&&*=0{}\ar@{-}[u]^c &&&&
}
$$
In particular we have $EC_{0,k} = C_{k+1}$. The trees $EC_{n,1}$ are called \emph{binary extended corollas}. The root horn of the extended corolla is the union of all faces except the
face obtained by chopping off the root {vertex}.
\end{definition}

\begin{thm} \label{THMFIBRANT}
For a dendroidal set $D$ {the following statement} are equivalent. 
\begin{enumerate}\itemsep 0pt
\item\label{eins} $D$ is fibrant in the stable model structure.
\item\label{zwei} $D$ is dendroidal Kan and admits fillers for all root horns of extended corollas $EC_{n,1}$.
\item\label{drei} $D$ is dendroidal Kan and admits fillers for all root horns of extended corollas $EC_{n,k}$.
\item\label{vier} $D$ is fully Kan.
\end{enumerate}
\end{thm}

We will prove Theorem \ref{THMFIBRANT} at the end of the paper. More precisely the equivalence of \eqref{eins} and \eqref{zwei} is Proposition \ref{PROP1}. The equivalence of \eqref{zwei} and \eqref{drei} is Proposition \ref{PROP2} and the equivalence of \eqref{drei} and \eqref{vier} is in Proposition \ref{PROP3}.



\section{Equivalence to connective spectra}\label{sec:equivalence}

Let $\E_\infty \in \dSet$ be a cofibrant resolution of the terminal object in $\dSet$. We furthermore assume that $\E_\infty$ has the property that 
the underlying space $i^* \E_\infty$ is equal to the terminal object $\Delta[0] \in \sSet$. 
The existence of such an object can be easily seen, e.g. using the small object argument (note that the cofibrant objects are the same 
in all three model structures on $\dSet$ that we consider).
In the following we denote $E_\infty := hc\tau_d(\E_\infty)$ which is an operad enriched over simplicial sets. Here 
\begin{equation*}
hc\tau_d: \dSet \to \text{s}\Oper
\end{equation*}
is the left adjoint to the homotopy coherent nerve functor, see \cite{MC11}. This is the functor that implements the aforementioned Quillen equivalence between dendroidal sets 
(with the Cisinski-Moerdijk model structure) and topological operads.
The operad $E_\infty$ is then cofibrant, has one colour and the property that each space of operations is contractible. Thus it is indeed an $E_\infty$-operad in the classical terminology. Therefore for each $E_\infty$-algebra $X$ in $\sSet$, the set $\pi_0(X)$ inherits the structure of an abelian monoid. Such an algebra $X$ is called \emph{group-like} if $\pi_0(X)$ is an abelian group, i.e. there exist inverses for each element.

Now denote by $\EsSet$ the category of $E_\infty$-algebras in simplicial sets. Recall from \cite[Section 3]{Heuts2} that there is an adjoint pair
$ \xymatrix{ St: \dSet_{/ \E_\infty} \ar@<0.3ex>[r] & \EsSet: Un \ar@<0.7ex>[l]}$
where $\dSet_{/ \E_\infty}$ denotes the category of dendroidal sets over $\E_\infty$. We do not repeat the definition of $St$ here since we need the formula only for a few particular simple cases and for these
cases we give the result explicitly.
\begin{example}
\begin{itemize}
\item
The $E_\infty$-algebra $St(\eta \to \mathcal{E}_\infty)$ is the free $E_\infty$-algebra on one generator, which we denote by $Fr(a)$ where $a$ is the generator.
\item
An object in  $\dSet_{/ \E_\infty}$ of the form $p: \Omega[C_2] \to \mathcal{E_\infty}$ encodes a binary operation $-\cdot_p-$ in the operad $E_\infty$.
Then $St(p)$ is the $E_\infty$-algebra freely generated by two generators $a$,$b$ and the square $\Delta[1] \times \Delta[1]$ subject to the relation 
that $a \cdot_p b \sim (1,1) \in \Delta[1]\times\Delta[1]$. We write this as
\begin{equation*}
St(\Omega[C_2] \to \mathcal{E_\infty}) = \frac{Fr(a,b,\Delta[1]^2)}{a\cdot_p b \sim (1,1)}.
\end{equation*}
\item
The three inclusions $\eta \to \Omega[C_2]$ induce maps $St(\eta \to \E_\infty) \to St(p)$. As usual we let $a,b$ be the leaves of the tree $C_2$ and $c$ the root. The first two maps are simply given by 
\begin{equation*}
Fr(a) \to Fr(a,b,\Delta[1]^2)/_\sim \quad a \mapsto a \qquad \text{and} \qquad Fr(b) \to Fr(a,b,\Delta[1]^2)/_\sim \quad b \mapsto b
\end{equation*}
The third map $Fr(c) \to Fr(a,b,\Delta[1]^2 ) /_\sim $ is given by sending $c$ to $(0,0) \in \Delta[1]^2$. Note that this third map is obviously homotopic to the map sending $c$ to $(1,1) = a \cdot_p b$.
\end{itemize}
\end{example}

The functor $P(D) := D \times \E_\infty$ induces a further adjoint pair
$ \xymatrix{P : \dSet \ar@<0.3ex>[r] & \dSet_{/ \E_\infty}: \Gamma \ar@<0.7ex>[l]}$.
Composing the two pairs $(St,Un)$ and $(P,\Gamma)$ we obtain  an adjunction
\begin{equation}\label{adjunction}
\xymatrix{ St_{\times \E_\infty}: \dSet \ar@<0.3ex>[r] & \EsSet: Un_\Gamma \ar@<0.7ex>[l]}
\end{equation}
Moreover $\EsSet$ carries a left proper, simplicial model structure where weak equivalences and fibrations are just weak equivalences and fibrations of the underlying space of an $E_\infty$-algebra, see \cite[Theorem 4.3. and Proposition 5.3]{spitzweckPhD} or 
\cite{berger2003axiomatic}. For this model structure and the covariant model structure on dendroidal sets the above adjunction \eqref{adjunction} is in fact a Quillen equivalence as shown by Heuts \cite{Heuts2}
\footnote{ Note that Heuts in fact uses a slightly different variant where $P$ is a right Quillen functor (instead of left Quillen). But if a right Quillen equivalence happens to be a left Quillen functor as well, then this left Quillen functor is also an equivalence. Thus Heuts' results immediately imply the claimed fact.}.

\begin{lemma}\label{unfibrant}
Let $X$ be a fibrant $E_\infty$-space. Then $X$ is group-like if and only if $Un_\Gamma(X) \in \dSet$ is fully Kan.
\end{lemma}
\begin{proof}
The condition that $Un_\Gamma(X)$ is fully Kan is by Theorem \ref{THMFIBRANT} equivalent to the map
\begin{equation*}
s^*: \sHom(\Omega[C_2], Un_\Gamma(X)) \to \sHom(\eta_a \sqcup \eta_c, Un_\Gamma(X))
\end{equation*}
being a weak equivalence of simplicial sets. By the Quillen equivalence \eqref{adjunction} and the fact that $\Omega[C_2]$ is cofibrant the space $\sHom(\Omega[C_2], Un_\Gamma(X))$ is homotopy equivalent to
the space $\sHom(St(\Omega[C_2] \times \E_\infty \to \E_\infty), X)$. {We can choose} a morphism
$p: \Omega[C_2] \to \E_\infty$ {(and this choice is essentially unique) because} $\Omega[C_2]$ is cofibrant and $\E_\infty \to *$ is a trivial fibration. In the covariant model structure on $\dSet_{/\E_\infty}$ (see \cite[Section 2]{Heuts2}) the objects $\Omega[C_2] \times \E_\infty \to \E_\infty$ and $\Omega[C_2] \to \E_\infty$ are cofibrant and equivalent. Cofibrancy is immediate and the fact that they are equivalent follows since the forgetful functor to dendroidal sets is a left Quillen equivalence and $\Omega[C_2] \simeq \Omega[C_2] \times \E_\infty$ in $\dSet$. 
Therefore $St(\Omega[C_2] \times \E_\infty \to \E_\infty)$ is weakly equivalent to $St(\Omega[C_2] \to \E_\infty)$ in $\EsSet$. Together we have the following weak equivalence of spaces  $$\sHom(\Omega[C_2], Un_\Gamma(X)) {\simeq} \sHom(St(\Omega[C_2]\to \E_\infty), X).$$

The same reasoning yields a weak equivalence $\sHom(\eta_a \sqcup \eta_c, Un_\Gamma(X)) {\simeq} \sHom(St(\eta_a \sqcup \eta_c \to \E_\infty), X)$ such that the diagram
\begin{equation}\label{square}
\xymatrix{
\sHom(\Omega[C_2], Un_\Gamma(X)) \ar[r]^{s^*} \ar[d]^\sim &  \sHom(\eta_a \sqcup \eta_c, Un_\Gamma(X))\ar[d]^\sim \\
\sHom(St(\Omega[C_2]\to \E_\infty), X) \ar[r]^{s^*} & \sHom(St(\eta_a \sqcup \eta_c \to \E_\infty), X)
}
\end{equation}
commutes.

Finally we use the fact that in the covariant model structure over $\E_\infty$ the leaf inclusion $i: \eta_a \sqcup \eta_b \to \Omega[C_2]$ is a weak equivalence. This implies that there is a further weak equivalence
$St(\eta_a \sqcup \eta_b \to \E_\infty) \xrightarrow{\sim} St(\Omega[C_2] \to \E_\infty)$. As remarked above, the straightening of $\eta \to \E_\infty$ is equal to $Fr(*)$,  the free $E_\infty$-algebra on one generator. Thus 
$St(\eta_a \sqcup \eta_b \to \E_\infty)$ is the coproduct of $Fr(a)$ and $Fr(b)$ which is isomorphic to $Fr(a,b)$ (here we used $a$ and $b$ instead of $*$ to label the generators). Then the above equivalence reads
$Fr(a,b) \xrightarrow{\sim} St(\Omega[C_2] \to \E_\infty)$. The root inclusion $r: \eta_c \to \Omega[C_2]$ induces a further map $r^*: Fr(c) = St(\eta_c \to \E_\infty) \to St(\Omega[C_2] \to \E_\infty)$ and using the explicit description of $St(p)$ given above we see that there is a homotopy commutative diagram
\begin{equation*}
\xymatrix{
St(\Omega[C_2] \to \E_\infty) && Fr(c) \ar[ll]_-{St(r)} \ar[ld]^f \\
& Fr(a,b) \ar[ul]^-{St(i)} &
}
\end{equation*}
where $f$ is defined as the map sending $c$ to the product $a \cdot_p b$. Thus the horn $s: \eta_a \sqcup \eta_c \to C_2$ fits in a homotopy commutative diagram
\begin{equation*}
\xymatrix{
St(\Omega[C_2]\to \E_\infty) && Fr(a,c) \ar[ll]_-{St(s)} \ar[ld]^{sh} \\
& Fr(a , b) \ar[ul]^{St(i)}
}
\end{equation*}
with the map $sh$ that sends $c$ to the binary product of $a$ and $b$ and $a$ to itself.

Putting the induced diagram together with diagram \eqref{square} we obtain the big diagram
\begin{equation}
\xymatrix{
\sHom(\Omega[C_2], Un_\Gamma(X)) \ar[r]^{s^*} \ar[d]^\sim &  \sHom(\eta_a \sqcup \eta_c, Un_\Gamma(X))\ar[d]^\sim \\
\sHom(St(\Omega[C_2]\to \E_\infty), X) \ar[r]^{s^*} \ar[d]^{\sim} & \sHom(Fr(a,c), X) \\
\sHom(Fr(a , b) , X) \ar[ur]^{sh*} &
}
\end{equation}
in which all the vertical arrows are weak equivalences. This shows that $Un_\Gamma(X)$ is fully Kan if and only if
$sh^*: \sHom(Fr(a , b) , X) \to  \sHom(Fr(a,c), X)$ is a weak equivalence. But we clearly have that the domain and codomain of this map are given by $X \times X$. Thus the map in question is given by the shear map
\begin{equation*}
Sh: X \times X \to X \times X \qquad (x,y) \mapsto (x,x \cdot_{p} y)
\end{equation*}
where $x \cdot_{p} y$ is the composition of $x$ and $y$ using the binary operation given by $hc\tau_d(p): \Omega(C_2) \to E_\infty$

It remains to show that a fibrant $E_\infty$-space $X$ is group-like precisely when the shear map $Sh: X \times X \to X \times X$ is a weak homotopy equivalence. This is well known \cite[chapter III.4]{Whitehead}, but we include it for completeness. Assume first that the shear map is a weak equivalence. Then the induced shear map $\pi_0(X) \times \pi_0(X) \to \pi_0(X) \times \pi_0(X)$ is an isomorphism. This shows that $\pi_0(X)$ is a group, thus $X$ is group-like. Assume conversely that $X$ is group-like and $y \in X$ is a point in $X$. Then there is an inverse $y' \in X$ together with a path connecting $y' \cdot_p y$ to the point $1$. This induces a homotopy inverse for the map $R_y: X \to X$ given by right multiplication with $y$ (for the fixed binary operation).  Now the shear map is a map of fibre bundles
\begin{equation*}
\xymatrix{
X \times X  \ar[rr]^{Sh} \ar[rd]_{pr_1} && X \times X \ar[ld]^{pr_1} \\
&X&
}
\end{equation*}
Thus the fact that it is over each point $y \in X$ a weak equivalence as shown above already implies that the shear map is a weak equivalence.
\end{proof}

Lemma \ref{unfibrant} shows that fully Kan dendroidal sets correspond to group-like $E_\infty$-spaces. We want to turn this into a statement about model structures. Therefore we need a model structure on $\EsSet$ where the fibrant objects are precisely the group-like $E_\infty$-spaces. 

\begin{proposition}
There is a left proper, combinatorial model structure on $\EsSet$ where the fibrant objects are precisely the fibrant, group-like $E_\infty$-spaces and which is a left Bousfield localization of the standard model structure on $\EsSet$. We call it the \emph{group-completion} model structure.
\end{proposition}
\begin{proof}
Since the model category of $\EsSet$ is left proper, simplicial and combinatorial the existence follows from general existence results provided that we can characterize the property of being group-like as a lifting property against a set of morphisms. The proof of Lemma \ref{unfibrant} already contains the argument, namely let the set consist of one map from the free $E_\infty$-algebra on two generators to itself given by the shear map (actually there is one shear map for each binary operation in $E_\infty$, but we simply pick one out).
\end{proof}

It is well known that group-like $E_\infty$-spaces model all connective spectra by the use of a delooping machine, see \cite{may74}. More precisely the $\infty$-category of group-like $E_\infty$-spaces obtained from the group-completion model structure is equivalent as an $\infty$-category to the $\infty$-category of connective spectra, see e.g. \cite[Remark 5.1.3.17]{HigherAlgebra}.

\begin{thm}\label{thm_equivalence}
The stable model structure on dendroidal sets is Quillen equivalent to the group-completion model structure on $\EsSet$ by the adjunction \eqref{adjunction}. Thus the stable model structure on dendroidal sets is a model for connective spectra in the sense that there is an equivalence of $\infty$-categories.
\end{thm}

The theorem follows from Lemma \ref{unfibrant} and the following more general statement about left Bousfield localizations and Quillen equivalences. Recall from \cite[Definition 1.3.]{Barwick07} that a combinatorial model category is called \emph{tractable} if it admits a set of generating cofibrations and generating trivial cofibrations with cofibrant domains and codomains. It turns out that it suffices to check this for generating cofibrations  \cite[Corollary 1.12.]{Barwick07}. Thus all model structures on dendroidal sets are clearly tractable.

\begin{lemma}
Let $C$ and $D$ be simplicial model categories with $C$ tractable and a (not necessarily simplicial) Quillen equivalence 
\begin{equation*}
\xymatrix{ L: C \ar@<0.3ex>[r] & D: R \ar@<0.7ex>[l]}
\end{equation*}
Moreover let $C'$ and $D'$ be left Bousfield localizations of $C$ and $D$ repectively. Assume $R$ has the property that a fibrant object $d \in D$ is fibrant in $D'$ if and only if $R(d)$ is fibrant in $C'$.

Then $(L \dashv R)$ is a Quillen equivalence between $C'$ and $D'$.
\end{lemma}

\begin{proof}
For simplicity we will refer to the model structures on $C$ and $D$ as the global model structures and to the model structures corresponding to $C'$ and $D'$ as the local model structures. First we have to show that the pair $(L,R)$ induces a Quillen adjunction in the local model structures. We will show that $L$ preserves local cofibrations and trivial cofibrations. Since local and global cofibrations are the same, this is true for cofibrations.  Thus we need to show it for trivial cofibrations and it follows by standard arguments if we can show it for generating trivial cofibrations.  Thus let $i: a \to b$ be a generating locally trivial cofibration in $C$. Now we can assume that $a$ and $b$ are cofibrant since $C$ is tractable. Then the induced morphism $\sHom(b, c) \to \sHom(a,c)$ on mapping spaces is a weak equivalence for every locally fibrant object $c \in C$. In particular for $c = R(d)$ with $d \in D$ locally fibrant. Now we use that there are weak equivalences $\sHom(b, R(d) ) \cong \sHom( Lb, d)$ and 
$\sHom(a, R(d) ) \cong \sHom( La, d)$ of simplicial sets which stem from the fact that the pair $(L,R)$ induces an adjunction of $\infty$-categories. {This shows} that the induced morphism $\sHom(Lb, d) \to \sHom(La, d)$ is a weak equivalence for every locally fibrant object $d \in D$. This shows that $La \to Lb$ is a local weak equivalence.

It remains to show that $(L,R)$ is a Quillen equivalence in the local model structures. Therefore it {suffices} to show that the right derived functor
\begin{equation*}
R': Ho(D') \to Ho(C')
\end{equation*}
is an equivalence of categories. {Since} $D'$ and $C'$ are Bousfield localizations $Ho(C')$ is a full reflective subcategory of $Ho(C)$ and correspondingly for $D$ and $D'$. Moreover, there is a commuting square
\begin{equation*}
\xymatrix{
Ho(D') \ar[r]^{R'} \ar@{^(->}[d]& Ho(C') \ar@{^(->}[d] \\
Ho(D)  \ar[r]^{R} & Ho(C)
}
\end{equation*}
Since $R$ is an equivalence {it} follows that $R'$ is fully faithful. In order to show that $R'$ is essentially surjective pick an {object} $c$ in $Ho(C')$ 
represented by a locally fibrant object {$c$} of $C$. Since $R$ is essentially surjective we find an element $d \in D$ which is globally fibrant such that $R(d)$ is equivalent to $c$ in $Ho(C)$. But this implies that $R(d)$ is also locally fibrant (i.e. lies in $Ho(C')$) since this is a property that is invariant under weak equivalences in Bousfield localizations. Therefore we conclude that $d$ is locally fibrant from the assumption on $R$. This shows that $R'$ is essentially surjective, hence an equivalence of categories.
\end{proof}
The fact that the stable model structure is equivalent to connective spectra has the important consequence that a cofibre sequence in this model structure is also a fibre sequence, which is well-known for connective spectra (note that the converse is not true in connective spectra, but in spectra).
\begin{corollary}
Let $X \to Y \to Z$ be a cofibre sequence of dendroidal sets in any of the considered model structures. Then
$$ i^*X_K \to i^* Y_K \to i^* Z_K $$
 is a fibre sequence of simplicial sets. Here $(-)_K$ denotes a fully Kan (fibrant) replacement.
\end{corollary}
\begin{proof}
Since the stable model structure on dendroidal sets is a Bousfield localization of the other model structures we see that a cofibre sequence in any model structure is also a cofibre sequence in the stable model structure. But then it is also a fibre sequence as remarked above. The functor $i^*$ is right Quillen, as shown in Theorem \ref{simplicial}. Thus it sends fibre sequences in $\dSet$ to fibre sequences in $\sSet$, which concludes the proof.
\end{proof}

\section{Proof of Theorem \ref{THMFIBRANT}, part I}\label{partI}

Recall from Definition \ref{ext_corolla} the notion of binary extended corollas.
Also recall from \cite{Heuts1} that the weakly saturated class generated by non-root horns of arbitrary trees is called the class of left anodynes. The weakly saturated class generated by inner horn inclusions of arbitrary trees is called the class of inner anodynes. Analogously we set:
\begin{definition}\label{binextleft}
The weakly saturated class generated by non-root horns of all trees and root horns of binary extended corollas is called the class of \emph{binary extended left anodynes}.
\end{definition}

\begin{proposition}\label{PROP1}
A dendroidal set $D$ is stably fibrant if and only if $D$ is a dendroidal Kan complex and it admits fillers for all root horns of {binary} extended corollas $EC_{n,1}$.
\end{proposition}
\begin{proof} 
We will show in Lemma \ref{partIlemma1} that a stably fibrant dendroidal set $D$ admits lifts against the root horn inclusion of $EC_{n,1}$.

Conversely, assume that $D$ is {a} dendroidal Kan {complex} and admits {lifts against} the root horn {inclusions} of $EC_{n,1}$. Then $D$ clearly admits lifts against all binary extended left anodyne morphisms. In Lemma \ref{partIlemma2} we show that the inclusion
\begin{equation*}
\Big( \Lambda^b [C_2] \otimes \Omega[L_n]\Big) ~ \cup ~ \Big( \Omega[C_2] \otimes \partial \Omega[L_n] \Big) \longrightarrow \Omega[C_2] \otimes \Omega[L_n])
\end{equation*}
is binary extended left anodyne. {This implies that} $D$ is stably fibrant.
\end{proof}

In the rest of the paper we prove {some technical lemmas} and for this we fix some terminology. 
We denote the leaves of the corolla $C_2$ by $a$ and $b$ and its root {edge} by $c$.
We denote the edges of the linear tree $L_n$ by $0,1,...,n$ {as indicated in the picture}
\begin{equation*}
\xymatrix@R=10pt@C=12pt{
&&&&\\
&&*=0{\bullet}\ar@{-}[u]^0&&\\
&&*=0{}\ar@{-}[u]^1&&\\
L_n=&&\dots&&\\
&&*=0{\bullet}\ar@{-}[u]^{n-1}&&\\
&&*=0{}\ar@{-}[u]^n&&
}
\end{equation*}
We denote the edges in the tensor product $\Omega[C_2]\otimes\Omega[L_n]$ by $a_i, b_i, c_i$ instead of $(a,i), (b,i), (c,i)$ and we {let $T_k$ for $k=0,1,...,n$ be} the {unique} shuffle of $\Omega[C_2]\otimes\Omega[L_n]$ 
that has the edges $a_k, b_k$ and $c_k$:
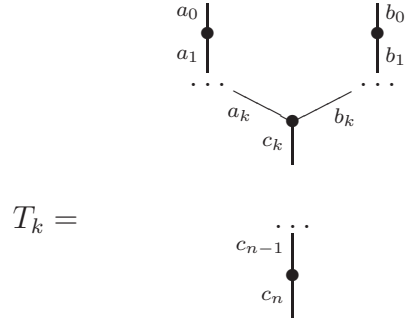
\begin{equation*} \xymatrix@R=10pt@C=12pt{
&&&&&&& \\
& *=0{\bullet} \ar@{-}[u]^{a_0} && *=0{\bullet}\ar@{-}[u]_{b_0} &&&& \\
& \dots \ar@{-}[u]^{a_1} && \dots \ar@{-}[u]_{b_1} &&&& \\
&&*=0{\bullet} \ar@{-}[ul]^{a_k} \ar@{-}[ur]_{b_k} &&&&&\\
&&*=0{}\ar@{-}[u]^{c_k} &&&&& \\
T_k = \qquad && \dots &&&& \\
&&*=0{\bullet}\ar@{-}[u]^{c_{n-1}} &&&&& \\
&&*=0{}\ar@{-}[u]^{c_n} &&&&& \\
}
\end{equation*}
{We also use the notation}
\begin{equation*}
D_iT_j=\left\{ 
\begin{array}{ll}
\partial_{a_i}\partial_{b_i} T_j, & i<j, \\ \partial_{c_i}T_j, & i>j.
\end{array}\right. 
\end{equation*}
We denote the subtrees of a shuffle as sequences of its edges with indices in the ascending order (since there is no danger of ambiguity).
For example we denote the following tree
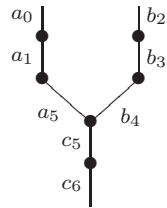
\begin{equation} \label{ex_subtree}
 \xymatrix@R=10pt@C=12pt{
&&&&&&& \\
& *=0{\bullet} \ar@{-}[u]^{a_0} && *=0{\bullet}\ar@{-}[u]_{b_2} &&&& \\
& *=0{\bullet} \ar@{-}[u]^{a_1} && *=0{\bullet} \ar@{-}[u]_{b_3} &&&& \\
&&*=0{\bullet} \ar@{-}[ul]^{a_5} \ar@{-}[ur]_{b_4} &&&&&\\
&&*=0{\bullet}\ar@{-}[u]^{c_5} &&&&& \\
&&*=0{}\ar@{-}[u]^{c_6} &&&&&
}
\end{equation}
by $(a_0,a_1,a_5,b_2,b_3,b_4,c_5,c_6)$.

{
We denote 
\begin{itemize}
\item by $\pi_i$ the unique dendrex of $\Omega[T_n]$ represented by a subtree with edges $b_n, c_n$ and $a_j$ for all $j\neq i$, for $i=0,...,n-1$;
\item by  $\pi_n$ the unique dendrex represented by $(a_0,...,a_{n-1}, b_{n-1}, c_{n-1})$ of $\Omega[T_{n-1}]$;
\item by $\alpha_n$ the unique dendrex represented by $(a_0,...,a_{n-1}, b_{n-1}, b_n, c_n)$ of $\Omega[T_n]$;
\item by $\sigma_j \alpha_n$ the degeneracy of $\alpha$ with respect to $a_j$, for $j=0,1,...,n-1$;
\item by $\beta_n$ the unique dendrex represented by $(a_0,...,a_{n-1},b_{n-1},c_{n-1},c_n)$ of $\Omega[T_{n-1}]$;
\item by $\gamma_n$ the unique dendrex represented by $(a_0,...,a_{n},b_{n},c_n)$  of $\Omega[T_n]$.
\end{itemize}
}

We denote the edges of the binary extended corolla as in the following picture: 
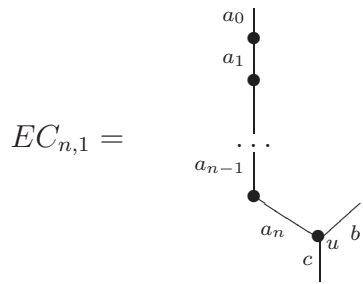
\begin{equation*}
\xymatrix@R=10pt@C=12pt{
&&&&&&&\\
& *=0{\bullet} \ar@{-}[u]^{a_0} && \\
& *=0{\bullet} \ar@{-}[u]^{a_1} && \\
EC_{n,1} = \qquad & \dots \ar@{-}[u] && \\
& *=0{\bullet} \ar@{-}[u]^{a_{n-1}} && \\
&&*=0{\,\, \bullet_u} \ar@{-}[ur]_{b}\ar@{-}[ul]^{a_n}&\\
&&*=0{}\ar@{-}[u]^c &
}
\end{equation*}

The colours of the tensor product $\Omega[EC_{n,1}] \otimes \Omega[L_1]$  will be denoted 
by $a_0,...,a_n$,$b$,$c$, $a'_0,...,a'_n$,$b'$,$c'$ and the operations are denoted accordingly. There are $n+1$ shuffles $E_0, E_1,..., E_n$ where $E_i$ is the unique shuffle that has $a_i$ and $a'_i$ for $i=0,...,n$ and one more shuffle $F$ which has $c$ and $c'$. For example we have the following shuffles
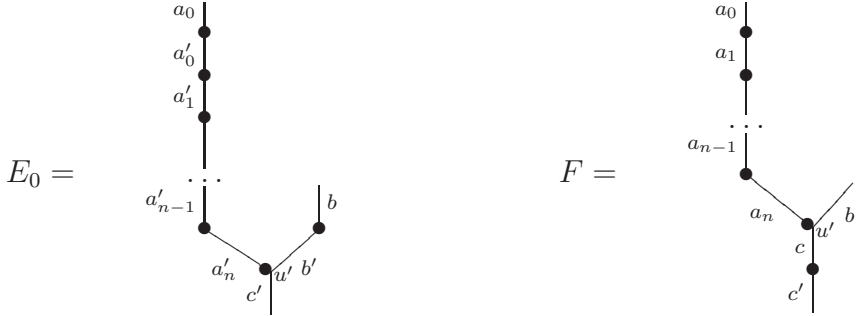
\begin{equation*}
\xymatrix@R=10pt@C=12pt{
&&&&&&&\\
& *=0{\bullet} \ar@{-}[u]^{a_0} && \\
& *=0{\bullet} \ar@{-}[u]^{a'_0} && \\
& *=0{\bullet} \ar@{-}[u]^{a'_1} && \\
E_0= \qquad & \dots \ar@{-}[u] && \\
& *=0{\bullet} \ar@{-}[u]^{a'_{n-1}} && *=0{\bullet} \ar@{-}[u]_{b} \\
&&*=0{\,\, \bullet_{u'}} \ar@{-}[ur]_{b'}\ar@{-}[ul]^{a'_n}&\\
&&*=0{}\ar@{-}[u]^{c'} & 
}
\quad
\xymatrix@R=10pt@C=12pt{
&&&&&&&\\
& *=0{\bullet} \ar@{-}[u]^{a_0} && \\
& *=0{\bullet} \ar@{-}[u]^{a_1} && \\
& \dots \ar@{-}[u] && \\
F= \qquad & *=0{\bullet} \ar@{-}[u]^{a_{n-1}} &&  \\
&&*=0{\,\, \bullet_{u'}} \ar@{-}[ur]_{b}\ar@{-}[ul]^{a_n}&\\
&&*=0{\bullet}\ar@{-}[u]^{c} & \\
&&*=0{}\ar@{-}[u]^{c'} &
}
\end{equation*}

\begin{lemma}\label{partIlemma1}
A stably fibrant dendroidal set $D$ admits lifts against the root horn inclusion $i \colon \Lambda^u [EC_{n,1}] \to \Omega[EC_{n,1}]$ of the binary extended corolla.
\end{lemma}

\begin{proof}
Let $D$ be a stably fibrant dendroidal set. By definition $D$ is a dendroidal Kan complex and admits lifts against the maps
\begin{equation*}
\Big( \Lambda^b [C_2] \otimes \Omega[L_n]\Big) ~ \cup ~ \Big( \Omega[C_2] \otimes \partial \Omega[L_n] \Big) \longrightarrow \Omega[C_2] \otimes \Omega[L_n]
\end{equation*}
for all $n \geq 0$.
{Note that the inclusion $\Lambda^a [C_2]\to \Omega[C_2]$ is isomorphic to the inclusion $\Lambda^b[C_2] \to \Omega[C_2]$. Hence $D$ also admits {lifts} against the maps
\begin{equation*}
\Big( \Lambda^a [C_2] \otimes \Omega[L_n]\Big) ~ \cup ~ \Big( \Omega[C_2] \otimes \partial \Omega[L_n] \Big) \longrightarrow \Omega[C_2] \otimes \Omega[L_n]
\end{equation*}
for all $n \geq 0$.
}
Consider the following pushout square 
\begin{equation*}
\xymatrix{\Big( \Lambda^a [C_2] \otimes \Omega[L_n]\Big) ~ \cup ~ \Big( \Omega[C_2] \otimes \partial \Omega[L_n] \Big)  \ar[r] \ar[d] & \Lambda^u[EC_{n,1}] \ar[d]^k \\ \Omega[C_2] \otimes \Omega[L_n]\ar[r]^l & P}
\end{equation*}
where the left vertical map is {the} inclusion and the top horizontal map is the unique map which maps $a_i$ to $a_i$, $b_i$ to $b$ and $c_i$ to $c$ for $i=0,1,...,n$. It follows that $D$ also admits a lift against the map $u \colon \Lambda^u [EC_{n,1}] \to P$. 
{
We can factor $k$ as a composition $k=pj$ of the inclusion 
\begin{equation*}
j \colon \Lambda^u EC_{n,1} \cong \Lambda^u EC_{n,1} \otimes\{1\} \to \Big(\Lambda^u EC_{n,1}\otimes \Omega[L_1] \Big) \cup \Big( EC_{n,1}\otimes\{0\} \Big)
\end{equation*}
and the map 
\begin{equation*}
p \colon \Big(\Lambda^u EC_{n,1}\otimes \Omega[L_1] \Big) \cup \Big( EC_{n,1}\otimes\{0\} \Big) \to P
\end{equation*} 
which we now describe explicitly. 
}

{ The colours of $P$ can be identified with $a_0,...,a_n,b$ and $c$.
The map $p$ is determined by the image of $EC_{n,1}\otimes\{0\} $ and compatibly chosen images of all the shuffles of $\Lambda^u EC_{n,1}\otimes \Omega[L_1]$, i.e. of $\partial_{a_i}F, i=0,1...,n$,  $\partial_{a_i}E_j, i=0,...,j$ and $\partial_{a'_i}E_j, i=j,...,n$ for all $j=0,1,...,n$. }

{
Concretely, we send 
\begin{itemize}
\item $EC_{n,1}\otimes \{0\}$ to $l(\gamma_n)$, 
\item $\partial_{a_n}F$ to $l(\beta_n)$, 
\item $\partial_{a'_n}E_j$ to $l(\sigma_j\alpha_n)$ for $j=0,1,...,n-1$, 
\item and all other shuffles to the corresponding degeneracy of $\pi_i$.  
\end{itemize}
One can easily verify that these conditions are compatible in $P$ and hence $p$ is well-defined.   
}
{
Now we can prove the statement of the lemma. So let us assume a map $f \colon \Lambda^u EC_{n,1} \to D$ is given. We want to prove that there is a lift $\bar{f} \colon EC_{n,1} \to D$ such that $f = \bar{f}  i$.  
}
{
By the above considerations we know that $D$ admits a lift $g : P \to  D$ such that  $f  = gk $ and hence $f$ factors also through $\Big(\Lambda^u EC_{n,1}\otimes \Omega[L_1] \Big) \cup \Big( EC_{n,1}\otimes\{0\} \Big)$ as a composition of $j$ and $gp$. We get the following commutative diagram 
\begin{equation*}
\xymatrix{\Lambda^u EC_{n,1} \otimes\{1\}  \ar[rr] \ar[d]_i &&  \Big(\Lambda^u EC_{n,1}\otimes \Omega[L_1] \Big) \cup \Big( EC_{n,1}\otimes\{0\} \Big) \ar[r] \ar[d]  & D \\ EC_{n,1}\otimes \{1\} \ar[rr] && EC_{n,1}\otimes \Omega[L_1]  &}
\end{equation*}
where the top horizontal maps are $j$ and $gp$ respectively and all other maps are obvious inclusions. 
Since $D$ is a dendroidal Kan complex it admits a lift against left anodynes and the right vertical inclusion $\Big(\Lambda^u EC_{n,1}\otimes \Omega[L_1] \Big) \cup \Big( EC_{n,1}\otimes\{0\} \Big) \to EC_{n,1}\otimes \Omega[L_1]$ is left anodyne because the covariant model structure is simplicial. Hence there is a lift $EC_{n,1}\otimes \Omega[L_1] \to D$ which, when precomposed with the inclusion $EC_{n,1}\otimes \{1\} \to EC_{n,1}\otimes \Omega[L_1]$, gives the desired lift $\bar{f}$. This finishes the proof. 
}
\end{proof}

\begin{lemma}\label{partIlemma2}
The pushout product of the map $s:\Lambda^b [C_2] \to \Omega[C_2]$ with a simplex boundary inclusion
\begin{equation*}
\Big( \Lambda^b [C_2] \otimes \Omega[L_n]\Big) ~ \cup ~ \Big( \Omega[C_2] \otimes \partial \Omega[L_n] \Big) \longrightarrow \Omega[C_2] \otimes \Omega[L_n]
\end{equation*}
is a binary extended left anodyne map.
\end{lemma}

\begin{proof}
The case $n=0$ is just the case of the inclusion $\Lambda^b [C_2] \to \Omega[C_2]$.
\medskip

Fix $n\geq 1$. We set $A_0:=\Lambda^b [C_2]\otimes \Omega[L_n] \coprod_{\Lambda^b[C_2]\otimes \partial \Omega[L_n]} \Omega[C_2]\otimes \partial \Omega[L_n]$.
Note that $A_0$ is the union of all $\Omega[D_iT_j]$ and of the chains $\eta_a\otimes \Omega[L_n] $ and $\eta_c\otimes \Omega[L_n]$.
We define dendroidal sets $A_{k}=A_{k-1}\cup \Omega[T_{k-1}]$ for $k=1,...,n+1$. So we have decomposed the map from the lemma into a composition of inclusions
\begin{equation*}
A_0 \subset A_1 \subset ... \subset A_{n-1} \subset A_n \subset A_{n+1}.
\end{equation*}
We will show that $A_k\to A_{k+1}$ is inner anodyne for $k=0,...,n-1$ and binary extended left anodyne for $k=n$.
Note that $A_{n+1}=\Omega[C_2]\otimes \Omega[L_n]$, so the inclusion $A_0\to \Omega[C_2]\otimes \Omega[L_n]$ is binary extended left anodyne as a composition of such maps.

\medskip

{\textbf{Case} $k=0$. }
The faces $\partial_{c_i}\Omega[T_0]$ {of $T_0$} are equal to $\Omega[D_iT_0]$ for all $i>0$. The outer leaf face of $T_0$ is equal to $\eta_c\otimes \Omega[L_n]$.
The remaining face $\partial_{c_0}\Omega[T_0]$ is in $A_1$, but not in $A_0$
so we have a pushout diagram
$$\xymatrix{\Lambda^{c_0}[T_0] \ar[r] \ar[d] & A_0 \ar[d] \\ \Omega[T_0] \ar[r] & A_{1}.}$$
Since inner anodyne extensions are closed under pushouts it follows that $A_0\to A_1$ is inner anodyne.

\medskip

\textbf{Case} $0<k<n$.
We now construct a further filtration
$$A_k=B^k_0 \subset B^k_1 \subset \cdots \subset B^k_{k+2}=A_{k+1}$$
as follows. Informally speaking, we add representables of subtrees of $T_k$ by the number of
vertices starting from the minimal ones which are not contained in {$A_{k}$}.
More precisely, set $B^k_0:=A_k$ and for $l=1,...,k+2$ let $B^k_l$ be the union of $B^k_{l-1}$ and all the representables of trees
$(a_{j_1},..., a_{j_q}, b_{i_1},..., b_{i_p}, c_k,..., c_n)$ with $q+p=l+k$ and $\{j_1, ..., j_q, i_1,..., i_p\}=\{0,1,...,k\}$.
{An an example of such a tree for $k=5, l=1, p=q=3$ and $n=6$ is given page 15. }

{For} $p+q=k+1$ and the tree $U= (a_{j_1},..., a_{j_q}, b_{i_1},..., b_{i_p}, c_k,...,c_n)$ we have an inclusion $\Lambda^{c_k}[U]\subset A_0=B_0^k$ because
$\partial_{c_i} \Omega[U] \subset \Omega[D_iT_k]$ for $i>k$,  $\partial_{a_j} \Omega[U]\subset \Omega[D_jT_k]$ for $j\in\{j_1,...,j_q\}$ and
$\partial_{b_i} \Omega[U]\subset \Omega[D_iT_k]$ for $i\in\{i_1,...,i_p\}$.  Also note that $\partial_{c_k}\Omega[U]$ is not contained in $A_0$.

For $p+q=k+l, l\geq 2$ and the tree $U=(a_{j_1},..., a_{j_q}, b_{i_1},..., b_{i_p}, c_k,..., c_n)$ we have an inclusion $\Lambda^{c_k}[U]\subset B_{l-1}^k$.
Indeed, for $j\in \{j_1, ..., j_q\}$, $\partial_{a_j}\Omega[U]\subset B^k_{l-1}$ by definition if $j\in\{i_1,...,i_p\}$ and $\partial_{a_j}\Omega[U]\subset A_{k-1}\subset B^k_{l-1}$ if $j\not\in\{i_1,...,i_p\}$.
Similarly,  $\partial_{b_i}\Omega[U] \subset B^k_{l-1}$ for $i\in\{i_1,...,i_p\}$ and $\partial_{c_i}\Omega[U]\subset \Omega[D_iT_k] \subset A_0$ for $i>k$.
The remaining face $\partial_{c_k}\Omega[U]$ is not contained $B^k_{l-1}$.

We conclude that the map $B^k_{l-1}\to B^k_l$ is inner anodyne for $l=1,..., k+2$ because it is {the pushout} of the inner anodyne map
$$\coprod_{q+p=k+l} \Lambda^{c_k}[U] \to \coprod_{q+p=k+l} \Omega[U]$$ where the coproduct is taken over all subtrees
$U=(a_{j_1}, a_{j_2},..., a_{j_q}, b_{i_1},..., b_{i_p}, c_k,..., c_n)$ of $T_k$ such that $q+p=k+l$ and $\{j_1, ..., j_q, i_1,..., i_p\}=\{0,1,...,k\}$.

\medskip

\textbf{Case} $k=n$. 
Note that  faces of the shuffle $T_n$ are
\begin{itemize}
\item $\partial_{b_i}T_n=(a_0,...,a_n, b_0,...,\widehat{b_i},...,b_n, c_n)$, $i=0,...,n$;
\item $\partial_{a_j}T_n=(a_0,...,\widehat{a_j},...,a_n, b_0,..., b_n, c_n)$, $j=0,...,n$.
\end{itemize}
{Our strategy goes as follows. First, we form} the union of $A_{n-1}$ with all $\partial_{b_i}\Omega[T_n], i=0,...n-1$. {Second, we} consider the union with all proper subsets of $\partial_{b_n}\Omega[T_n]$ that
contain edges $a_0$ and $a_n$. {Third, we consider} the union with $\partial_{a_j}\Omega[T_n], j=1,...n$ and then with $\partial_{a_0}\Omega[T_n]$. In the last step we use the horn inclusion
$\Lambda^{b_n}[T_n]\subset \Omega[T_n]$.
Thus we start with a filtration
\begin{equation*}
A_n=P_0 \subset \cdots \subset P_{p-1} \subset P_p\subset \cdots \subset P_{n},
\end{equation*}
{where $P_p$ is the union of $P_{p-1}$ with the representables of the trees of the form $(a_0,...,a_n,$ $b_{i_1},...,b_{i_p}, c_n)$ for $p=1,...,n-1$}.
Also, we define $P_{n}$ as the union of $P_{n-1}$ with $\partial_{b_i}\Omega[T_n]$ {for all $i=1,2,...,n-1$ (but not for $i=n$)}.
{Let us} show that the maps $P_{p-1}\to P_p$ are left anodyne for $p=1,2,...,n$.

\begin{itemize}
\item \textbf{Case} $p=1$.
For $i\in\{0,1,...,n\}$ and $V_i=(a_0,...,a_n, b_i, c_n)$ all the faces of $\Omega[V_i]$, except $\partial_{a_i}\Omega[V_i]$, are in $P_0=A_n$. The map $P_0 \to P_1$ is left anodyne as a pushout of the map $\coprod_{i=0}^n \Lambda^{a_i}[V_i]\to \coprod_{i=0}^n \Omega[V_i]$.
$$  \xymatrix@R=10pt@C=12pt{
&&&&&&& \\
& *=0{\bullet} \ar@{-}[u]^{a_0} &&&&&& \\
V_i= & \dots \ar@{-}[u]^{a_1} &&&&&& \\
& *=0{\bullet} \ar@{-}[u]^{a_{n-1}} &&&&&& \\
&&*=0{\bullet} \ar@{-}[ul]^{a_n} \ar@{-}[ur]_{b_i} &&&&&\\
&&*=0{}\ar@{-}[u]^{c_n} &&&&& \\
}
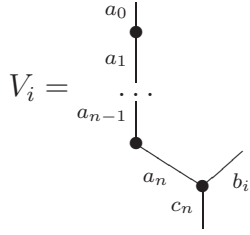
$$
\item \textbf{Case} $p\leq n-1$.
We give a further filtration $$P_{p-1}=Q^p_0  \subset Q^p_1 \subset \cdots\subset Q^p_m \subset \cdots \subset Q^p_p=P_p$$

Let $Q^p_m$ be the union of $Q^p_{m-1}$ with $\Omega[U]$ for all the trees of the form $$U=(a_{j_1},...,a_{j_q}, b_{i_1},...,b_{i_p},c_n), ~ q+p=n+m$$ such that there is a subset $I\subseteq \{i_1,...,i_{p-1}\}$ with $\{j_1,...,j_q\}=\{0,1,...,n\}~\setminus~I$. Note that $i_p\in \{j_1,...,j_q\}$. We show that the inclusions $Q^p_{m-1}\to Q^p_m$ are left anodyne for all $m=1,2,...,p-1$.
For a fixed $m$ and such a tree $U$ the faces of $\Omega[U]$ are all in $Q^p_{m-1}$ except for $\partial_{a_{i_p}}\Omega[U]$.
More precisely, the faces $\partial_{b_i}\Omega[U]$ are all in $P_{p-1}$, the faces $\partial_{a_j}\Omega[U]$ are in $A_0$ if $j\not\in \{i_1,...,i_p\}$ and
in $Q^p_{m-1}$ by definition if $j\in \{i_1,...,i_p\}$.

We conclude that $Q_{m-1}^p\to Q^p_m$ is left anodyne as a pushout of the left anodyne map $\coprod_{} \Lambda^{a_{i_p}}[U]\to \coprod_{} \Omega[U]$, where the coproduct is taken over trees $U$ described above.
We have $P_p=Q^p_p$, so $P_{p-1}\to P_p$ is also left anodyne.

\item \textbf{Case} $p=n$.
Here we do a slight modification of the previous argument.
Let $Q^n_0:=P_{n-1}$ and for $m=1,...,n-1$ let $Q^{n}_m$ be the union of $Q^n_{m-1}$ with $\Omega[U_i]$ for the trees of the form
$$U_i=(a_{i_1},...,a_{i_m}, a_n, b_0,...,\hat{b_i},...,b_n, c_n), i\neq n$$ or of the form  $$U_n=(a_0, a_{i_1},..., a_{i_{m-1}}, a_n,b_0,...,b_{n-1},c_n).$$

Let $Q^n_n$ be the union of $Q^n_{n-1}$ with $\partial_{b_i}\Omega[T_n]$ {for all $i=1,2,...,n-1$ (but not for $i=n$)}.

A similar argument to the one given above (using horns $\Lambda^{a_n}[U_i], i\neq n$ and $\Lambda^{a_0}[U_n]$) shows that the maps $Q^n_{m-1}\to Q^n_m$ are left anodyne for all $m=1,...,n$.
Since $P_n=Q_n^n$ we have proven that $P_{n-1}\to P_n$ is left anodyne and hence $A_n\to P_n$ is left anodyne.

\end{itemize}

Next, we add $\partial_{a_i}\Omega[T_n]$ for $i=1,2,...,n$ to the union.
Let us denote the only binary vertex of the tree $W=(a_0,b_0,...,b_n,c_n)$ by $v$.
Let $P_{n+1}= P_{n}\cup \Omega[W]$.  Then the map $P_n\to P_{n+1}$ is \emph{binary extended left anodyne} because it is a pushout of the map
$ \Lambda^{v}[W] \to \Omega[W]$.

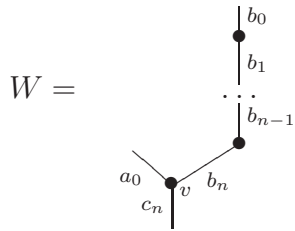
\begin{equation*}  
\xymatrix@R=10pt@C=12pt{
&&&&&&& \\
&  && *=0{\bullet}\ar@{-}[u]_{b_0} &&&& \\
W= &  && \dots \ar@{-}[u]_{b_1} &&&& \\
&  && *=0{\bullet}\ar@{-}[u]_{b_{n-1}} &&&& \\
&&*=0{\,\,\bullet_v} \ar@{-}[ul]^{a_0} \ar@{-}[ur]_{b_n} &&&&&\\
&&*=0{}\ar@{-}[u]^{c_n} &&&&&
}
\end{equation*}

For each $q=2,...,n$ we define $P_{n+q}$ as the union of $P_{n+q-1}$ and the 
representables of the trees of the form $Z_q=(a_0, 
a_{i_1},...,a_{i_q},b_0,...,b_n,c_n)$. The {inclusion} $P_{n+q-1}\to P_{n+q}$ is left 
anodyne as the pushout of $\coprod \Lambda^{a_0}[Z_q]\to \coprod_{} 
\Omega[Z_q]$.

The dendroidal set $P_{2n}$ contains $\partial_{a_i}\Omega[T_n], i=1,...,n$. Furthermore, {all} faces of $\partial_{a_0}\Omega[T_n]$ except for $\partial_{b_n}\partial_{a_0}\Omega[T_n]$ are in $P_{2n}$. Let $P_{2n+1}=P_{2n}\cup \partial_{a_0}\Omega[T_n]$. Then $P_{2n}\to P_{2n+1}$ is inner anodyne as the pushout of $\Lambda^{b_n}\partial_{a_0}[T_n]\to \partial_{a_0}\Omega[T_n]$.
From this we conclude that $A_n\to P_{2n+1}$ is binary extended left anodyne. All the faces of $\Omega[T_n]$ {except $\partial_{b_n}\Omega[T_n]$} are in $P_{2n+1}$, so $P_{2n+1}\to A_{n+1}$ is left anodyne as the pushout of the map $\Lambda^{b_n}[T_n]\to \Omega[T_n]$. Hence $A_n\to A_{n+1}$ is binary extended left anodyne, which finishes the proof.

\end{proof}

\section{Proof of Theorem \ref{THMFIBRANT}, part II}\label{partII}

{In this section we compare lifts against binary extended corollas and all extended corollas.}
\begin{proposition}\label{PROP2}
Let $D \in \dSet$ be a dendroidal Kan complex. Then $D$ admits fillers for all root horns of binary extended corollas $EC_{n,1}$ if and only if $D$ admits fillers for all root horns of arbitrary extended corollas $EC_{n,k}$.
\end{proposition}
\begin{proof}
One direction is a special case and thus trivial. Hence assume $D$ admits fillers for all root horns of extended corollas $EC_{n,1}$. Then $D$ admits lifts against all binary extended left anodynes (see Definition \ref{binextleft}). We need to show that $D$ admits lifts against the root horn inclusion
$\Lambda^u [EC_{n,k}] \to \Omega[EC_{n,k}]$. By Lemma \ref{partIIlemma1} we find a tree $T$ and a morphism $\Omega[EC_{n,k}] \to \Omega[T]$ such that the composition $\Lambda^u [EC_{n,k}] \to \Omega[T]$ is binary extended left anodyne. Thus given a morphism $\Lambda^u [EC_{n,k}] \to D$ we {can} find a filler $\Omega[T] \to D$. But the composition $\Omega[EC_{n,k}] \to \Omega[T] \to D$ is then the desired lift.
\end{proof}

\begin{lemma}\label{partIIlemma1}
Consider the inclusion of the root horn of the extended corolla  $\Lambda^u [EC_{n,k}] \to \Omega[EC_{n,k}]$. There is a tree $T$ and a morphism $\Omega[EC_{n,k}] \to \Omega[T]$ such that the composition  $\Lambda^u [EC_{n,k}]\to \Omega[T]$ is a binary extended left anodyne map.
\end{lemma}
\begin{proof}

We use the labels for edges of the extended corolla $EC_{n,k}$ as given in the Definition \ref{ext_corolla} and in addition we denote its root vertex by $u$.
Now consider the tree $T$
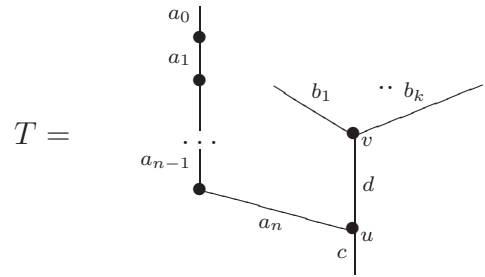
\begin{equation*} \xymatrix@R=10pt@C=12pt{
&&&&&&&\\
& *=0{\bullet} \ar@{-}[u]^{a_0} &&&&&& \\
& *=0{\bullet} \ar@{-}[u]^{a_1} && \qquad .. &&&& \\
T = \qquad & \dots \ar@{-}[u] && *=0{~\bullet_v}\ar@{-}[ul]_{b_1}\ar@{-}[urr]^{b_k} &&&& \\
& *=0{\bullet} \ar@{-}[u]^{a_{n-1}} &&&&&& \\
&&&*=0{~\bullet_u}\ar@{-}[uu]_{d}\ar@{-}[ull]^{a_n}&&&&\\
&&&*=0{}\ar@{-}[u]^c &&&&
}
\end{equation*}
There is an obvious morphism $\Omega[EC_{n,k}] \to \Omega[T]$. We will show that the composition $\Lambda^u [EC_{n,k}]\to \Omega[T]$  is binary extended left anodyne. \\

\noindent We set $E_0:=\Lambda^u [EC_{n,k}]$. Let $C_k$ be a corolla with root $d$ and leaves $b_1,...,b_k$

\begin{equation*}  \xymatrix@R=10pt@C=12pt{
&&&&&&& \\
C_k = \qquad &&&*=0{\bullet}\ar@{-}[ul]_{b_2}\ar@{-}[ull]^{b_1}\ar@{-}[urrr]^{b_k}&&&&\\
&&&*=0{}\ar@{-}[u]^d &&&&
}
\end{equation*}
Set $E_1:=E_0\cup \Omega[C_k]$ which is a subobject of $\Omega[T]$. The map $E_0\to E_1$ {is a pushout of the map $\coprod_{i=1}^k \eta_{b_i} \to \Omega[C_k]$, so it is left anodyne by definition.}

As a next step consider subtrees of $T$ which are of the form
$$  \xymatrix@R=10pt@C=12pt{
&&&&&&&\\
& *=0{\bullet} \ar@{-}[u]^{a_{i_0}} &&&&&& \\
& *=0{\bullet} \ar@{-}[u]^{a_{i_1}} && \qquad \dots &&&& \\
T_{i_0,...,i_l} = \qquad & \dots \ar@{-}[u] && *=0{~\bullet_v}\ar@{-}[ul]_{b_1}\ar@{-}[urr]^{b_k} &&&& \\
& *=0{\bullet} \ar@{-}[u]^{a_{i_{l-1}}} &&&&&& \\
&&&*=0{~\bullet_u}\ar@{-}[uu]_{d}\ar@{-}[ull]^{a_{i_l}}&&&&\\
&&&*=0{}\ar@{-}[u]^c &&&&
}
$$
for $\{i_0,...i_l\}\subset\{0,1,...,n\}$ and $l\leq n-1$. We define dendroidal sets $E_{l+2}$ as the union of $E_{l+1}$ and all representables $\Omega[T_{i_0,...,i_l}]$ for $\{i_0,...i_l\}\subset\{0,1,...,n\}$ and $0\leq l\leq n-1$. Thus we get a filtration
\begin{equation}\label{fillemma71}
\Lambda^u [EC_{n,k}] = E_0 \subset E_1 \subset E_2 \subset ... \subset E_{n+1} \subset \Omega[T].
\end{equation}
For a fixed $l\leq n-1$ and a subset $\{i_0,...,i_l\}$ the inner face $\partial_d\Omega[T_{i_0,...,i_l}]$ is contained in $E_0$ and the faces $\partial_{a_j}\Omega[T_{i_0,...,i_l}]$ are contained in $E_{l+1}$ for every $j\in\{i_0,...,i_n\}$ (and for $l=0$ the face $\partial_u\Omega[T_{i_0}]$ is in $E_1$).

Since $\partial_v\Omega[T_{i_0,...,i_l}]$ is not in $E_{l+1}$ we have the following pushout diagram
\begin{equation*}
\xymatrix{\coprod \Lambda^{v} [T_{i_0,...,i_l}]\ar[r] \ar[d] & E_{l+1} \ar[d] \\ \coprod \Omega[T_{i_0,...,i_l}] \ar[r] & E_{l+2}}
\end{equation*}
where the coproduct varies over all possible $(i_0,..., i_l)$. This shows that $E_{l+1}\to E_{l+2}$ is left anodyne.
From this we conclude that all maps in the above filtration \eqref{fillemma71} except for the last inclusion are left anodyne and therefore also the  map $E_0\to E_{n+1}$ is left anodyne. \\

\noindent We proceed by observing that for the tree
$$  \xymatrix@R=10pt@C=12pt{
&&&&&&& \\
& *=0{\bullet} \ar@{-}[u]^{a_0} &&&&&& \\
V= & \dots \ar@{-}[u]^{a_1} &&&&&& \\
& *=0{\bullet} \ar@{-}[u]^{a_{n-1}} &&&&&& \\
&&*=0{~\bullet_u} \ar@{-}[ul]^{a_n} \ar@{-}[ur]_{d} &&&&&\\
&&*=0{}\ar@{-}[u]^{c} &&&&& \\
}
$$
all faces of $\Omega[V]$ are in $E_{n+1}$ except $\partial_u \Omega[V]$. Notice that $E_{n+1} \cup \Omega[V]=\Lambda^d [T]$.
The map $E_{n+1}\to \Lambda^d[T]$ is the pushout of the binary extended left anodyne map $\Lambda^{u} [V]\to \Omega[V]$, so it is binary extended left anodyne. 
Finally, since ${\Lambda^d\Omega[T]} \to \Omega[T]$ is inner anodyne, we conclude that $E_0\to \Omega[T]$ is binary extended left anodyne.
\end{proof}

\section{Proof of Theorem \ref{THMFIBRANT}, part III}\label{partIII}

{Similarly to Definition \ref{binextleft} of binary extended left anodynes we define two more classes.}

\begin{definition}
The weakly saturated class generated by non-root horns of all trees and root horns of extended corollas is called the class of \emph{extended left anodynes}. The weakly saturated class generated by all horn inclusions of trees is called the class of \emph{outer anodynes}.
\end{definition}

It would be more natural to call outer anodynes simply anodynes since it also includes the inner anodynes. But, in order to make the distinction clearer, we use the term outer anodynes here. By definition we have inclusions
\begin{align*}
\{\text{inner anodynes}\}  \subset \{\text{left anodynes}\}  \subset \{\text{binary ext. left anodynes}\} \\ 
\subset \{\text{ext. left anodynes}\} \subset \{\text{outer anodynes}\} 
\end{align*}
All of these inclusions are proper, except for the last one. {In the following proposition we show that the last inclusion is actually an equality.}

\begin{proposition}\label{PROP3}
The class of extended left anodynes and the class of outer anodynes coincide.
In particular, a dendroidal set $D$ admits lifts against all non-root horns and root horns of extended corollas if and only if it is fully Kan.
\end{proposition}

\begin{proof}
By the above inclusion of saturated classes it suffices to show that {every} root horn inclusion is contained in the class of extended left anodynes. A root horn for a tree exists only if this tree is {obtained by grafting a smaller tree} on a corolla. {We give the proof of this technical statement in Lemma \ref{partIIIlemma2}.}
\end{proof}

Before we can prove the crucial lemma we need to introduce some terminology. Recall from \cite{MW09} that \emph{a top face map} is an outer face map with respect to a top vertex and \emph{an initial segment} of a tree is a {subtree} obtained by composition of top face maps.
For example, the  {tree $V$ is an initial segment of the tree $T$ in the following picture.}
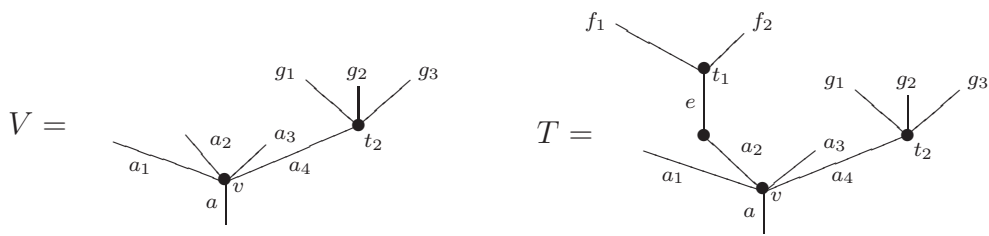
\begin{equation*}  \xymatrix@R=10pt@C=12pt{
&&&&&&\\
&&&_{g_1}&_{g_2}&_{g_3}& \\
V = \qquad & &&_{a_3}& *=0{\bullet} \ar@{-}[ul]\ar@{-}[u]\ar@{-}[ur]\ar@{-}_{t_2}&& \\
&&*=0{~\bullet_v}\ar@{-}[ull]^{a_1}\ar@{-}[ul]_{a_2}\ar@{-}[ur]\ar@{-}[urr]_{a_4}&&&&\\
&&*=0{}\ar@{-}[u]^a &&&&
} \quad 
\xymatrix@R=10pt@C=12pt{
_{f_1}&&_{f_2}&&&&\\
& *=0{~~\bullet_{t_1}} \ar@{-}[ul]  \ar@{-}[ur] &&_{g_1}&_{g_2}&_{g_3}& \\
T = \qquad  & *=0{\bullet} \ar@{-}[u]^{e} &&_{a_3}& *=0{\bullet} \ar@{-}[ul]\ar@{-}[u]\ar@{-}[ur]\ar@{-}_{t_2}&& \\
&&*=0{~\bullet_v}\ar@{-}[ull]^{a_1}\ar@{-}[ul]_{a_2}\ar@{-}[ur]\ar@{-}[urr]_{a_4}&&&&\\
&&*=0{}\ar@{-}[u]^a &&&&
}
\end{equation*}

\begin{definition}
{A subtree which is a composition of an initial segment followed by exactly $k$ inner face maps is called an \emph{initial subtree of codimension $k$}.}
\end{definition}
By definition, every initial segment is an initial subtree of codimension zero.
An example of {an initial} subtree of codimension 2 of the {above} tree $T$  is 
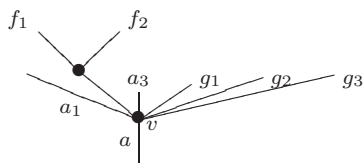
\begin{equation*}  \xymatrix@R=10pt@C=12pt{
& _{f_1}&&_{f_2} &&&&&& \\
&& *=0{\bullet} \ar@{-}[ul]  \ar@{-}[ur] &_{a_3}&_{g_1}&_{g_2}&_{g_3}& \\
&&&*=0{~\bullet_v}\ar@{-}[ull]^{a_1}\ar@{-}[ul]\ar@{-}[u] \ar@{-}[ur] \ar@{-}[urr] \ar@{-}[urrr]  &&&&\\
&&&*=0{}\ar@{-}[u]^a &&&&
}
\end{equation*}

\begin{lemma}[Codimension argument]
Let $T$ be a tree and $v$ a vertex of $T$. Let $V$ be the maximal initial segment of $T$ for which the input edges $d_1,...,d_p$ of $v$ are leaves. Let $X_T$ be a subobject of $\Omega[T]$ defined in the following way:
If $V$ has at least two vertices, then $X_T$ is the 
union of the following dendroidal sets
\begin{itemize}\setlength{\itemsep}{-0.5ex}
\item the representable $\Omega[V]$,
\item the inner faces $\partial_e\Omega[T]$ for all inner edges $e$ of $V$,
\item the outer faces $\partial_u\Omega[T]$ for vertices $u$ of $V$, $u\neq v$.
\end{itemize}
If $V$ has exactly one vertex, then $X_T$ is the union of the following dendroidal sets
\begin{itemize}\setlength{\itemsep}{-0.5ex}
\item the representable $\Omega[V]$,
\item the representable of the maximal subtree of $T$ having $d_i$ as root for $i=1,..,p$.
\end{itemize}
Then the inclusion $X_T \to \Omega[T]$ is inner anodyne.

\end{lemma}

\begin{proof}
Let $|V|$ and $|T|$ {denote} the number of vertices of $V$ and $T$, respectively. Let $N=|T|-|V|+1$. We say that an initial subtree $S$ of codimension $k$ of $T$ containing $V$ is an \emph{$(n,k)$-subtree} if it has exactly $|V|-1+n$ vertices.
Note that $V$ is a $(1,0)$-subtree.

Denote $X_{1,0}:=X_T$. The strategy is to form an inner anodyne filtration consisting of dendroidal sets $X_{(n,k)}$ by considering unions of $X_{1,0}$ with {some} $(n,k)$-subtrees of $T$ for all $n, 1\leq n\leq N$ and $k, 0\leq k \leq {N-n}$.  

{
Before constructing this filtration, we form a set $\F_{n,k}$ of \emph{chosen} $(n,k)$-subtrees which we do not include in $X_{n,k}$ for each pair $(n,k)$.      
We start with the tree $T$ which is an $(N,0)$-subtree and we choose $\partial_{d_i}T$ for $i$ being minimal such that $d_i$ is an inner edge of $T$. The set $\F_{N-1,1}$ has only one element $\partial_{d_i}T$ and $\F_{N-1, 0}$ is empty. We proceed inductively by decreasing $n$ from $N$ to $1$. Each $(n+1,k-1)$-subtree $S$ which is not in $\F_{n+1,k-1}$ contains at least one inner edge $d_j$, $j\in \{1,...,p\}$ and we choose $\partial_{d_i} S$ for minimal such $i$ and put this $(n,k)$-subtree $\partial_{d_i} S$ in $\F_{n,k}$.  
}

Note that for $n=2, k\geq 1$ such a subtree $S$ has exactly $|V|+1$ vertices and only one inner face {$\partial_{d_i}S$ and that face belongs to $\F_{1,k}$.} {Hence $X_{1,0}=X_T$. We define $X_1=X_T$ and for $2\leq n\leq N$ we inductively define $X_{n,0}$ as the union of $X_{n-1}$ and the representables of all $(n,0)$-subtrees, $X_{n,k}$ as the union of $X_{n,k-1}$ and the representables of all $(n,k)$-subtrees that are not in $\F_{n,k}$ and dendroidal sets $X_n$ as the union $\ds \bigcup_{k=0}^{{N-n}} X_{n,k}$.}

The inclusions $X_{n-1}\to X_{n,0}$ are all inner anodyne because each of them is a pushout of the coproduct of inner horn inclusions.
More precisely, each $(n,0)$-subtree $S$ has faces which are in $X$ by definition, outer faces that are $(n-1,0)$-subtrees and hence are all in $X_{n-1}$, inner faces which are $(n-1,1)$-subtrees and by definition exactly one of them was chosen to be in $F_{n-1,1}$, so is not in $X_{n-1}$. Denote this inner face by $\partial_s S$.
We have the pushout diagram  (where the coproduct is taken over all $(n,0)$-subtrees)
$$\xymatrix{\coprod \Lambda^{s}[S] \ar[r] \ar[d] & X_{n-1} \ar[d] \\ \coprod \Omega[S] \ar[r] & X_{n,0}}$$
Note that the union of representables of $(n+1,k-1)$-subtrees and $X_{n+1,k-1}$ will also contain the representables of elements of $\F_{n,k}$ (since the elements of $\F_{n,k}$ will be faces of the $(n+1,k-1)$-subtrees). So $X_{n+1,k-1}$ will contain representables of all $(n,k)$-subtrees.
The inclusions $X_{n+1,k-1}\to X_{n+1,k}$ are similarly shown to be inner anodyne. Faces of an $(n+1,k)$-subtree are in $X$ or $(n,k)$-subtrees (and hence all in $X_{n+1,k-1}$ by the previous sentence) or $(n,k+1)$-subtrees (and hence all but one in $X_{n,k+1}\subset X_n\subset X_{n+1,k-1}$ by construction). We again have a horn inclusion with respect to the excluded face, and $X_{n+1,k-1}\to X_{n+1,k}$ is the pushout of the coproduct of these horn inclusions.
Finally, we have shown that the inclusion $X_T=X_1\subset X_2 \subset ... \subset X_N=\Omega[T]$ is left anodyne.
\end{proof}

\begin{definition}
For a {non-linear} tree $T$ the maximal subtree having non-unary root is unique and we call it the \emph{tree top} of $T$. {For a linear tree we say that its tree top is given by its unique leaf (i.e. it is isomorphic to $\eta$).}
The maximal initial segment of $T$ which is a linear tree is also unique and we call it the \emph{stem} of $T$. Note that $T$ is obtained by grafting the tree top of $T$ to the stem of $T$ and conversely the tree top is obtained from $T$ by chopping {off} the stem. 
\end{definition}

\noindent
For a fixed tree $T$ with root $r$ we define the tree $U_{T,q}$ obtained by grafting $T$ to the $(q+1)$-corolla with leaves $r,b_1,...,b_q$, the root $c$ and the root vertex $u$.
Let $T'$ be the tree that has one edge more than $T$ such that this edge, called $a'$, is the leaf of the stem of $T'$ (and the root of the tree top of $T'$).
Let $W=W_ {T,q}$ be the tree obtained by grafting $T'$ to the $(q+1)$-corolla with leaves $r,b_1,...,b_q$, the root $c$ and the root vertex $u$.

We will usually denote by $v$ the root vertex of the tree top of $T$ and the input edges of $v$ by  $d_1,...,d_p$. We will denote by $v'$ the vertex in $W$ having the output $a'$.
The edges of the stem of $T$ will be denoted $a_0,...,a_l$ with $a_i$ and $a_{i+1}$ being the input and the output of the same vertex for all $i=0,...,l-1$ (so $a_l$ is the root). {Here is one example. }

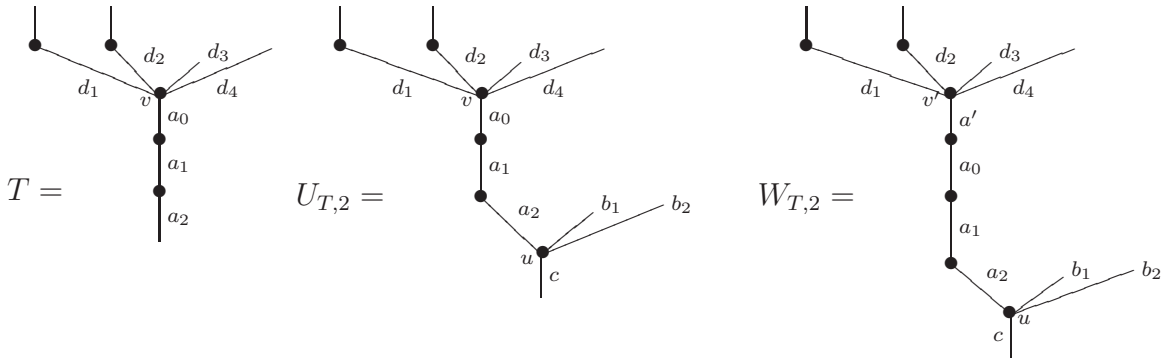
{
\begin{equation*}  
\xymatrix@R=10pt@C=12pt{
&&&&\\
*=0{\bullet} \ar@{-}[u] &*=0{\bullet} \ar@{-}[u] && _{d_3}   & \\
 && *=0{\!\!_{v}\bullet}\ar@{-}[ull]^{d_1}\ar@{-}[ul]_{d_2}\ar@{-}[ur]\ar@{-}[urr]_{d_4}&&\\
 && *=0{\bullet}\ar@{-}[u]_{a_0}&&\\
T = && *=0{\bullet}\ar@{-}[u]_{a_1}&&\\
 && \ar@{-}[u]_{a_2}&&
} 
\xymatrix@R=10pt@C=12pt{
&&&&&&\\
*=0{\bullet} \ar@{-}[u] &*=0{\bullet} \ar@{-}[u]&&_{d_3}&&& \\
&& *=0{\!\!_{v}\bullet}\ar@{-}[ull]^{d_1}\ar@{-}[ul]_{d_2}\ar@{-}[ur]\ar@{-}[urr]_{d_4}&&&&\\
&& *=0{\bullet}\ar@{-}[u]_{a_0}&&&&\\
U_{T,2} =  && *=0{\bullet}\ar@{-}[u]_{a_1}&&_{b_1}&_{b_2}&\\
&&& *=0{\!\!_u \bullet} \ar@{-}[ul]_{a_2}\ar@{-}[ur] \ar@{-}[urr] &&& \\
&&&*=0{}\ar@{-}[u]_c  &&&
} 
 \xymatrix@R=10pt@C=12pt{
&&&&&&\\
*=0{\bullet} \ar@{-}[u] &*=0{\bullet} \ar@{-}[u]&&_{d_3}&&& \\
&& *=0{\!\!\!\!_{v'} \bullet}\ar@{-}[ull]^{d_1}\ar@{-}[ul]_{d_2}\ar@{-}[ur]\ar@{-}[urr]_{d_4}&&&&\\
&& *=0{\bullet}\ar@{-}[u]_{a'}&&&&\\
W_{T,2} = && *=0{ \bullet}\ar@{-}[u]_{a_0}&&&&\\
&& *=0{ \bullet}\ar@{-}[u]_{a_1}&&_{b_1}&_{b_2}&\\
&&& *=0{~\bullet_{u}} \ar@{-}[ul]_{a_2}\ar@{-}[ur] \ar@{-}[urr] &&& \\
&&&*=0{}\ar@{-}[u]^c &&&
}  
\end{equation*}
}

For a subset $J\subset \{0,1,...,l\}$ we denote by
\begin{itemize}
\item $U^0_J$ the unique subtree of $W$ containing the edges $d_1,...,d_p,a',b_1,...,b_q, c$ and $a_j, j\in J$.

\item $U'_J$ the maximal subtree of $W$ not containing the edges $a_j, j\in \{0,1,...,l\} \setminus J$. 

\item $T^0_J$ and $T'_J$ the root face of $U^0_J$ and $U'_J$, respectively.
\end{itemize}
Note that $T'_J$ contains the whole tree top of $T$, while $T^0_J$ only the non-unary root vertex of the tree top of $T$. 

\begin{lemma}\label{partIIIlemma2}
{Let $U$ be a tree whose root vertex $u$ is attached to exactly one inner edge. 
The inclusion $\Lambda^u [U] \to \Omega[U]$ is extended left anodyne.}
\end{lemma}
\begin{proof}
{There is a tree $T$ and a natural number $q \geq 0$ such that $U = U_{T,q}$.}
Let $N$ be the number of vertices of the tree top $S$ of $T$ and let $l$ be the number of vertices of the stem of $T$. We show the claim by induction on $N$.

{If $N=0$ the tree $T$ is linear and the claim holds by definition of extended left anodynes.}

Fix a tree top $S$ with $N$ vertices, $N\geq 1$, and assume that the claim holds for every tree {such that the corresponding tree top has} less than $N$ vertices.
We will prove that for fixed $S$ and for every $l$, the inclusion $\Lambda^u [U] \to \Omega[U]$ is extended left anodyne.
Since  $\Lambda^u [U] \to \Omega[U]$ is a retract of  $\Lambda^u [U]\to \Omega[W]$, it is enough to show that $\Lambda^u [U]\to\Omega[W]$ is
extended left anodyne. We divide the proof into four parts. \\

\noindent{\textbf{Step 1.}} We show that the inclusion $\Lambda^u [U]\to \cup_{j=0}^l \Omega[\partial_{a_j}W]$ is left anodyne.

We denote $B_0:=\Lambda^u[U]$.
Inductively, for all $1\leq k\leq {l+1}$, we define
\begin{align*}
A'_{k-1}:=B_{k-1}\cup \bigcup_{|J|=k-1}\Omega[T^0_J], \quad A_{k}:=A'_{k-1}\cup \bigcup_{|J|=k-1}\Omega[T'_J], \\
B'_{k-1}:=A_k\cup \bigcup_{|J|=k-1}\Omega[U^0_J],  \quad B_k:=B'_{k-1}\cup \bigcup_{|J|=k-1}\Omega[U'_J].
\end{align*}
Since $A'_0=A_0\cup \Omega[T^0_\emptyset]$ and $T^0_\emptyset$ is the $p$-corolla with inputs $d_1,...,d_p$ and root $a'$, the inclusion $A_0\to A'_0$ is the pushout of $\eta_{d_1}\cup ... \cup \eta_{d_p} \to \Omega[T^0_\emptyset]$ {and hence left anodyne.} \\
Let $k$ be such that $1\leq k\leq l$. The inclusion $A_k\to B'_{k-1}$ is left anodyne because it is the pushout of the coproduct of leaf horn inclusions $\coprod_{|J|=k} \Lambda^{v'} [U^0_J] \to \coprod_{|J|=k}\Omega[U^0_J]$. The inclusion $B_k\to A'_k$ is left anodyne because it is the pushout of the coproduct of leaf horn inclusions $\coprod_{|J|=k} \Lambda^{v'} [T^0_J] \to \coprod_{|J|=k}\Omega[T^0_J]$. \\
For all trees $T'_J, |J|=k-1$, and vertex $v'$ the codimension argument gives an inner anodyne $X_{T'_J}\to \Omega[T'_J]$. {Since $X_{T'_J}$ is exactly the intersection of $A'_{k-1}$ and $\Omega[T'_J]$, the inclusion $A'_{k-1}\to A_k$ is inner anodyne as the pushout of the coproduct $\coprod_{|J|=k-1} X_{T'_J}\to\coprod_{|J|=k-1}\Omega[T'_J]$. }
{Similarly, we use the codimension argument to show that $X_{U'_J}\to \Omega[U'_J]$ is inner anodyne. As $X_{U'_J}$ is the intersection of $\Omega[U'_J]$ and $B'_{k-1}$ the inclusion $B'_{k-1}\to B_k$ is inner anodyne as the pushout of the coproduct $\coprod_{|J|=k-1} X_{U'_J}\to \coprod_{|J|=k-1}\Omega[U'_J]$.} Note that {$B_{l+1}=\bigcup_{j=0}^l \Omega[\partial_{a_j}W]$, so this completes the first step. }\\

\noindent {\textbf{Step 2.} 
Let $V_0$ be the {unique} initial segment of $W$ for which $a'$ is a leaf. We define $D_0:=\Omega[V_0] \cup \bigcup_{j=0}^l \Omega[\partial_{a_j}W]$. The map $\bigcup_{j=0}^l \Omega[\partial_{a_j}W]\to D_0$ is \emph{extended left anodyne} because it is the pushout of outer root inclusion of $\Omega[V_0]$.}\\

\noindent{\textbf{Step 3.} For $1\leq n \leq N-1$, we define the set $\mathcal{V}_n$ of all the initial segments of $W$  with exactly $n+l+2$ vertices. Furthermore we inductively define dendroidal sets $D_n=D_{n-1}\cup \bigcup_{V\in \mathcal{V}_n} \Omega[V]$. Note that all such subtrees $V\in \mathcal{V}_n$ contain $a',a_0,...,a_l,b_1,...,b_q,c$ since they are initial and they have exactly $n$ vertices more than $V_0$.}
{The outer root horn inclusion for $V\in \mathcal{V}_n$ is extended left anodyne by the inductive hypothesis. The intersection of $\Omega[V]$ and $D_{n-1}$ is the horn $\Lambda^u [V]$ because the faces $\partial_{a_j}\Omega[V], j=0,1,...,l $ are in $B_{l+1}\subset D_0$ by the previous arguments, the face $\partial_{a'} V$ is in $A_0$, and other inner and outer leaf faces are in $D_{n-1}$ by definition. We conclude that the inclusion $D_{n-1}\to D_{n}$ is also extended left anodyne because it is the pushout of $\bigcup_{V\in \mathcal{V}_n}\Lambda^u(V) \to \bigcup_{V\in \mathcal{V}_n} \Omega[V]$. } 

Note that $D_{N-1}$ contains all the faces of $W$ except the outer root face $T'$ and $\partial_{a'} W =U$. We have so far proven that $\Lambda^u [U]\to D_{N-1}$ is extended left anodyne. \\

\noindent {\textbf{{Step 4.} } We show that $D_{N-1}\to \Omega[W]$ is inner anodyne.
The intersection of $\Omega[T']$ and $D_{n-1}$ is the inner horn $\Lambda^{a'}[T']$ because the inner face $\partial_{a'} T'=T$ is not in $D_{N-1}$ and }
\begin{itemize}
\item $\partial_{a_j}  T'$ is already in $A_l$;
\item $\partial_e T'$ for inner edges $e$ of the tree top $S$ are in $D_{N-1}$ because $D_{N-1}$ contains $\partial_e W$;
\item $\partial_t T'$ for top vertices $t$ of the tree top $S$ are in $D_{N-1}$ because $D_{N-1}$ contains $\partial_t W$.
\end{itemize}
So the map $D_{N-1}\to D_{N-1}\cup \Omega[T']=\Lambda^{a'} [W]$ is inner anodyne because it is a pushout of an inner horn inclusion.
Finally, $\Lambda^{a'} [W] \to \Omega[W]$ is inner anodyne and we have shown that the inclusion $A_0\to \Omega[W]$ is extended left anodyne.
\end{proof}

\bibliographystyle{alpha}
\bibliography{Alles}

\end{document}